\documentclass [11pt]{article}
\usepackage[]{amsmath,amssymb,amsfonts,latexsym,amsthm,enumerate}
\usepackage{mathrsfs}
\usepackage[numeric,initials,nobysame]{amsrefs}

\long\def\symbolfootnote[#1]#2{\begingroup%
\def\thefootnote{\fnsymbol{footnote}}\footnote[#1]{#2}\endgroup}

\def\aas{\textrm{a.a.s.}}

\def\R{\mathbb{R}}

\newif\ifhide
\newif\ifexpose
\hidefalse
\exposetrue

\def\poly{\textrm{poly}}
\def\erfi{\mathrm{erfi}}

\def\aut{\textrm{aut}}
\newcommand{\N}{\mathbb{N}}

\def\B{B}

\def\A{\mathcal{A}}

\def\E{\mathcal{E}}
\def\G{\mathcal{G}}

\def\Q{\mathcal{Q}}

\def\GG{\mathbb{G}}
\def\MM{\mathbb{M}}
\def\TF{\mathbb{TF}}
\def\expec#1{{\mathbb E[#1]}\,}

\def\eventcnd#1#2{\{#1\,|\,#2\}}

\def\event#1{\{#1\}}

\def\probcnd#1#2{\textrm{Pr}[#1\,|\,#2]}
\def\prob#1{\textrm{Pr}[#1]}

\def\floor#1{\lfloor #1 \rfloor}

\def\ceil#1{\lceil #1 \rceil}

\date{}
\title{Triangle-free subgraphs at the triangle-free process}
\author{{Guy Wolfovitz\thanks{Department of Computer Science, 
Haifa University, Haifa, Israel. Email address:
{\tt gwolfovi@cs.haifa.ac.il}.}}}
\newtheorem{theorem}{Theorem}[section]
\newtheorem{lemma}[theorem]{Lemma}
\newtheorem{claim}[theorem]{Claim}
\newtheorem{definition}{Definition}

\newtheorem{corollary}[theorem]{Corollary}
\newtheorem{proposition}[theorem]{Proposition}

\newtheorem{fact}[theorem]{Fact}
\renewcommand{\epsilon}{\varepsilon}

\DeclareMathOperator{\var}{Var}
\DeclareMathOperator{\Cov}{Cov}

\newtheoremstyle{upright}%
        {8pt plus2pt minus4pt}%
        {8pt plus2pt minus4pt}%
        {\upshape}%
        {}%
        {\bfseries}%
        {:}%
        {1em}%
        {}%
\theoremstyle{upright}
\newtheorem{remark}[theorem]{Remark}

\newcommand{\ignore}[1]{}
\begin{document}

\maketitle

\begin{abstract}
We consider the triangle-free process: given an integer $n$, start by taking a
uniformly random ordering of the edges of the complete $n$-vertex graph $K_n$.
Then, traverse the ordered edges and add each traversed edge to an (initially
empty) evolving graph - unless its addition creates a triangle.  We study the
evolving graph at around the time where $\Theta(n^{3/2 + \epsilon})$ edges have
been traversed for any fixed $\epsilon \in (0,10^{-10})$. At that time and for
any fixed triangle-free graph $F$, we give an asymptotically tight estimation
of the expected number of copies of $F$ in the evolving graph. For $F$ that is
balanced and have density smaller than $2$ (e.g., for $F$ that is a cycle of
length at least $4$), our argument also gives a tight concentration result for
the number of copies of $F$ in the evolving graph.
Our analysis combines Spencer's original branching process approach for
analysing the triangle-free process and the semi-random method.
\end{abstract}

\section{Introduction} \label{sec:1}
In this paper we consider the triangle-free process. This is a random greedy
process that generates a triangle-free graph as follows.  Given $n \in \N$,
take a uniformly random ordering of the edges of the complete $n$-vertex graph
$K_n$. Here, we take that ordering as follows.  Let $\beta : K_n \to [0,1]$ be
chosen uniformly at random; order the edges of $K_n$ according to their
\emph{birthtimes} $\beta(f)$ (which are all distinct with probability $1$),
starting with the edge whose birthtime is smallest.  Given the ordering,
traverse the ordered edges and add each traversed edge to an evolving
(initially empty) triangle-free graph, unless the addition of the edge creates
a triangle.
When all edges of $K_n$ have been exhausted, the process ends. Denote by
$\TF(n)$ the triangle-free graph which is the result of the above process.
Further, denote by $\TF(n,p)$ the intersection of $\TF(n)$ with $\{f : \beta(f)
\le p\}$.

For a graph $F$, let $X_F$ be the random variable that counts the number of
copies of $F$ in $\TF(n,p)$.  We use $e_F$ and $v_F$ to denote respectively the
number of edges and vertices in a graph $F$ and set $\aut(F)$ to be the number
of automorphisms of $F$. A graph $F$ is \emph{balanced} if $e_F/v_F \ge
e_H/v_H$ for all $H \subseteq F$ with $v_H \ge 1$. We say that an event holds
\emph{asymptotically almost surely} (\emph{\aas}) if the probability of the
event goes to $1$ as $n \to \infty$.  For $m_1 = m_1(n)$, $m_2 = m_2(n)$, we
write $m_1 \sim m_2$ if $m_1/m_2$ goes to $1$ as $n \to \infty$.  Let $\ln n$
denote the natural logarithm of $n$.  Our main result follows.
\begin{theorem}
\label{thm:main}
Fix a triangle-free graph $F$ and $\epsilon \in (0, 10^{-10})$. 
For some $p \sim n^{\epsilon - 1/2}$,
\begin{eqnarray*}
\expec{X_F} \sim \frac{v_F!}{\aut(F)}\binom{n}{v_F}\bigg(\frac{\ln n^\epsilon}{n}\bigg)^{e_F/2}.
\end{eqnarray*}
\end{theorem}
Our second result gives a concentration result for $X_F$, for certain fixed triangle-free
graphs $F$.
\begin{theorem}
\label{thm:main1}
Fix a balanced triangle-free graph $F$ with $e_F/v_F < 2$. Then there exists $0
< \epsilon_F \le 10^{-10}$ such that for all $\epsilon \in (0, \epsilon_F)$ the
following holds.  For some $p \sim n^{\epsilon - 1/2}$, a.a.s., 
\begin{eqnarray*}
X_F \sim \frac{v_F!}{\aut(F)}\binom{n}{v_F}\bigg(\frac{\ln n^\epsilon}{n}\bigg)^{e_F/2}.
\end{eqnarray*}
\end{theorem}

One interesting point worth making with respect to Theorem~\ref{thm:main1} is
this. Let $F$ be a balanced triangle-free graph with density $e_F / v_F < 2$.
Fix $\epsilon \in (0, \epsilon_F)$, where $\epsilon_F$ is as guaranteed to
exist by Theorem~\ref{thm:main1}.  Let $p \sim n^{\epsilon - 1/2}$ be as
guaranteed to exist by Theorem~\ref{thm:main1}.  Consider the random graph
$\GG(n,m)$, which is chosen uniformly at random from among those $n$-vertex
graphs with exactly $m := \floor{2^{-1}n^{3/2}\sqrt{\ln n^\epsilon}}$ edges.
Note that by Theorem~\ref{thm:main1}, $\TF(n,p)$ and $\GG(n,m)$ a.a.s. has
asymptotically the same number of edges.  This of course follows directly from
our choice of the parameter $m$. The point is that by standard techniques and
by Theorem~\ref{thm:main1}, we also have that a.a.s., the number of copies of
$F$ in $\GG(n,m)$ is asymptotically equal to the number of copies of $F$ in
$\TF(n,p)$. Furthermore, $\GG(n,m)$ is expected to contain many triangles, and
indeed it does contain many triangles a.a.s., whereas $\TF(n,p)$ contains no
triangles at all.  Therefore, one may argue, at least with respect to the
number of copies of fixed balanced triangle-free graphs with density strictly
less than $2$, that $\TF(n,p)$ ``looks like'' a uniformly random graph with $m$
edges--only that it has no triangles.  A similar point can be made with respect
to Theorem~\ref{thm:main}.

\subsection{Related results}
Erd\H{o}s, Suen and Winkler~\cite{ErdosSW95} were the first to consider the
triangle-free process. They proved that the number of edges in $\TF(n)$ is
a.a.s. bounded by $\Omega(n^{3/2})$ and $O(n^{3/2}\ln n)$.
Spencer~\cite{Spencer0a} showed that for every two reals $a_1, a_2 > 0$, there
exists $n_0$ such that the number of edges in $\TF(n)$ for $n \ge n_0$ is
expected to be at least $a_1 n^{3/2}$ and is a.a.s. at most $a_2 n^{3/2}\ln n$.
In the same paper, Spencer conjectured that the number of edges in $\TF(n)$ is
a.a.s. $\Theta(n^{3/2}\sqrt{\ln n})$. In a recent breakthrough, this conjecture
was proved valid by Bohman~\cite{Bohman}. We remark that
Theorem~\ref{thm:main1} generalizes Bohman's lower bound for the number of
edges in $\TF(n)$ and answers a question of Spencer~\cite{SpencerPrivate}.  We
discuss in some more details Bohman's result below.

Other results are known for the more general $H$-free process. In the $H$-free
process, instead of forbidding a triangle, one forbids the appearance of a copy
of~$H$. Let $\MM(H,n)$ be the graph produced by the $H$-free process.  There
are several results with regard to the number of edges in
$\MM(H,n)$~\cites{RucinskiW92, Boll00, OsthusT01, Bohman, Wol, BKeevash}. For a
graph $H \ne K_3$ that is strictly $2$-balanced, the best lower bounds (which
are probably optimal) on the number of edges in $\MM(H,n)$ are provided by
Bohman and Keevash~\cite{BKeevash}; the best upper bounds on the number of
edges in $\MM(H,n)$ are provided by Osthus and Taraz~\cite{OsthusT01} and are
within $\poly(\ln n)$ factors from the best lower bounds.

Lastly, in~\cites{ErdosSW95, Spencer0a, Bohman, BKeevash}, the authors consider
the independence number of $\MM(H,n)$ for some graphs $H$.  Most notable are
the results of Bohman~\cite{Bohman} and of Bohman and Keevash~\cite{BKeevash}.
Bohman studies the independence number of $\MM(H,n)$ for $H \in \{K_3, K_4\}$.
His results imply Kim's~\cite{Kim} celebrated lower bound on the off-diagonal
Ramsey number $r(3,t)$ and a new lower bound for $r(4,t)$.  Bohman and Keevash
extend Bohman's results for every $H$ that is strictly $2$-balanced. By that,
they obtain new lower bounds for the off-diagonal Ramsey numbers $r(s,t)$ for
every fixed $s \ge 5$.

\subsection{Comparison with Bohman's argument}
Bohman's analysis of the triangle-free process in~\cite{Bohman} shows that the
number of edges in $\TF(n)$ is a.a.s. $\Omega(n^{3/2}\sqrt{\ln n})$.
Theorem~\ref{thm:main1} generalizes this result in that it matches Bohman's
lower bound up to a constant and in addition provides an a.a.s. lower bound on
the number of copies of $F$ in $\TF(n)$, for every fixed $F$ that is a balanced
triangle-free graph with density less than $2$. Moreover, Bohman's result
implies a lower bound of $\Omega(n^{3/2}\sqrt{\ln n})$ on the expected number
of edges in $\TF(n)$.  Theorem~\ref{thm:main} generalizes this result in that
it matches Bohman's lower bound up to a constant and provides a lower bound on
the expected number of copies of $F$ in $\TF(n)$ for every fixed triangle-free
graph $F$. Below we shortly discuss and compare Bohman's argument and ours.

Bohman uses the differential equations method in order to analyse $\TF(n,p)$
for $p = n^{\epsilon - 1/2}$ and some fixed $\epsilon > 0$. The basic argument
can be described as follows.  First, a collection of random variables that
evolve throughout the random process is introduced and tracked throughout the
evolution of $\TF(n,p)$.  This collection includes, for example, the random
variable $|O_i|$, where $O_i$ denotes the set of edges that have not yet been
traversed by the process, and which can be added to the current graph without
forming a triangle, after exactly $i$ edges have been added to the evolving
graph.
Now, at certain times during the process (i.e., at those times in which new
edges are added to the evolving graph), the expected change in the values of
the random variables in the collection is expressed using the same set of
random variables. This allows one to express the random variables in the
collection using the solution to an autonomous system of ordinary differential
equations.  It is then shown that the random variables in the collection are
tightly concentrated around the trajectory given by the solution to this
system.  The particular solution to the system then implies that $|O_I|$ is
a.a.s. large for $I = \Omega(n^{3/2}\sqrt{\ln n})$.  This then implies the
a.a.s. lower bound on the number of edges in $\TF(n)$.

In comparison with the above, we analyse $\TF(n,p)$ for $p = n^{\epsilon -
1/2}$ and some fixed $\epsilon > 0$ using the original branching process
approach of Spencer~\cite{Spencer0a} together with the semi-random method.
These two are combined together using combinatorial arguments.  Apart from our
different approach for the analysis of the triangle-free process, our actual
argument is more direct, in the sense that we estimate directly the probability
that any fixed triangle-free graph $F$ is included in $\TF(n,p)$.  Doing so
allows us to infer the validity of our two main results using standard
techniques. 

We remark that in the course of our analysis, we track and show the
concentration of some random variables that in retrospect (and perhaps not
surprisingly) turned out to be essentially the same random variables as some of
those that were tracked by Bohman. We choose to keep this part of the proof
both for the sake of completeness and since it provides an alternative argument
for the concentration of these random variables.

Lastly, we note that exactly like Bohman's argument, our ideas can be
generalized so as to obtain results which are similar in spirit to our main
theorems for the more general $H$-free process for a large family of graphs
$H$.  Moreover, since our arguments allow us to reason about subgraphs other
than edges in the evolving graph, we can prove results of the following form:
``a.a.s. every set of $t$ vertices in $\MM(H,n)$ spans a copy of $F$'' for some
$t$ and some fixed graphs $F$.  In particular for $H = K_4$, we can use the
ideas presented in this paper in order to show that a.a.s. every set of $t =
O(n^{3/5}(\ln n)^{1/5})$ vertices in $\MM(K_4,n)$ spans a triangle.  This
implies an a.a.s. upper bound on the number of edges in $\MM(K_4,n)$ which
matches up to a constant Bohman's lower bound.  

\section{Preliminaries} \label{section:preliminaries}

\subsection{Notation}
As usual, for a natural number $a$, let $[a] := \{1, 2, \ldots, a\}$.  We write
$x = a(y \pm z)^b$ if it holds that $x \in [a(y-z)^b, a(y+z)^b]$. We also use
$a(y \pm z)^b$ to simply denote the interval $[a(y-z)^b, a(y+z)^b]$.  All
asymptotic notation in this paper is with respect to $n \to \infty$.  All
inequalities in this paper are valid only for $n \ge n_0$, for some
sufficiently large $n_0$ which we do not specify.  

\subsection{Azuma's inequality}
The following result is a version of Azuma's inequality~\cite{Hoeffding},
tailored for combinatorial applications (see e.g.~\cites{McD,uppertail}).  Let
$\alpha_1, \alpha_2, \ldots, \alpha_m$ be independent random variables with
$\alpha_i$ taking values in a set $A_i$. Let $\psi : A_1 \times A_2 \times
\ldots \times A_m \to \R$ satisfy the following Lipschitz condition: if two
vectors $\alpha, \alpha' \in A_1 \times A_2 \times \ldots \times A_m$ differ
only in the $i$th coordinate, then $|\psi(\alpha) - \psi(\alpha')| \le c_i$.
Then the random variable $X = \psi(\alpha_1, \alpha_2,\ldots, \alpha_m)$
satisfies for any $t \ge 0$,
\begin{eqnarray*} \prob{|X - \expec{X}| \ge t} \le
2\exp\bigg(-\frac{2t^2}{\sum_{i=1}^m c_i^2}\bigg).  \end{eqnarray*}
%

\section{Proof of Theorems~\ref{thm:main}~and~\ref{thm:main1}}
\label{section:alternative}
In this section we prove Theorems~\ref{thm:main}~and~\ref{thm:main1}, modulo
one technical result.  We begin by giving an alternative definition of the
triangle-free process.  Under this alternative definition, we state a result
(Theorem~\ref{thm:main:}) which trivially implies Theorem~\ref{thm:main}.  We
then use this result in order to prove Theorem~\ref{thm:main1}.  The rest of
the paper will then be devoted for proving the above mentioned result.

Fix once and for the rest of the paper $\epsilon \in (0, 10^{-10})$.  Define
$\delta := 1/\floor{n^\epsilon}$ and $I := \delta^{-2}$.  For every integer $i
\ge 0$ define a triangle-free graph $\TF_i$ as follows.  Initially, take
$\TF_0$ to be the empty graph over the vertex set of $K_n$ and set $\B_0 :=
\emptyset$.  Given $\TF_i$, define $\TF_{i+1}$ as follows.  Choose uniformly at
random a function $\beta_{i+1} : K_n \setminus \B_{\le i} \to [0, 1]$ where
$\B_{\le i} := \bigcup_{j \le i} \B_j$.  Let $\B_{i+1}$ be the set of edges $f$
for which the \emph{birthtime} $\beta_{i+1}(f)$ satisfies $\beta_{i+1}(f) <
\delta n^{-1/2}$.  Traverse the edges in $\B_{i+1}$ in order of their
birthtimes (starting with the edge whose birthtime is smallest), and add each
traversed edge to $\TF_i$, unless its addition creates a triangle.  Denote by
$\TF_{i+1}$ the graph thus produced.  Observe that $\TF_I$ has the same
distribution as $\TF(n, p)$ for some $p \sim n^{\epsilon - 1/2}$. 

Let $\Phi(x)$ be a function over the reals, whose derivative is denoted by
$\phi(x)$, and which is defined by $\phi(x) := \exp(-\Phi(x)^2)$ and $\Phi(0)
:= 0$.  This is a separable differential equation whose solution (taking into
account the initial value) is given implicitly by $\frac{\sqrt{\pi}}{2}
\erfi(\Phi(x)) = x$, where $\erfi(x)$ is the imaginary error function, given by
$\erfi(x) := \frac{2}{\sqrt{\pi}}\int_0^x \exp(t^2) dt$. We have that $\erfi(x)
\to \exp(x^2)/(\sqrt{\pi}x)$ as $x \to \infty$. Hence, it follows that $\Phi(x)
\to \sqrt{\ln x}$ as $x \to \infty$.

By the discussion above, linearity of expectation and the fact that the
number of copies of $F$ in $K_n$ is $\frac{v_F!}{\aut(F)}\binom{n}{v_F}$ the
following result trivially implies Theorem~\ref{thm:main}.
\begin{theorem} \label{thm:main:}
Let $F \subset K_n$ be a triangle-free graph of size $O(1)$. Then
\begin{eqnarray*}
\prob{F \subseteq \TF_I} \sim \bigg(\frac{\Phi(I\delta)}{\sqrt{n}}\bigg)^{e_F}.
\end{eqnarray*}
\end{theorem}
For a graph $F$, let $Y_F$ be the random variable that counts the number of
copies of $F$ in $\TF_I$. The following theorem clearly implies
Theorem~\ref{thm:main1}.
\begin{theorem}
Fix a balanced triangle-free graph $F$ with $e_F/v_F < 2$.  Then there exists
$0 < \epsilon_F \le 10^{-10}$ such that for all $\epsilon \in (0, \epsilon_F)$
the following holds. A.a.s.,
\begin{eqnarray*}
Y_F \sim \frac{v_F!}{\aut(F)} \binom{n}{v_F}\bigg(\frac{\Phi(I\delta)}{\sqrt{n}}\bigg)^{e_F}.
\end{eqnarray*}
\end{theorem}
\begin{proof}
Fix a balanced triangle-free graph $F$ with $e_F / v_F < 2$.  Assume $\epsilon
\in (0, \epsilon_F)$ for some $0 < \epsilon_F \le 10^{-10}$ sufficiently small
so that it satisfies our arguments below.  The number of copies of $F$ in $K_n$
is $\frac{v_F!}{\aut(F)} \binom{n}{v_F}$.  Therefore, by
Theorem~\ref{thm:main:},
\begin{eqnarray*}
\expec{Y_F} \sim \frac{v_F!}{\aut(F)}\binom{n}{v_F} \bigg(\frac{\Phi(I\delta)}{\sqrt{n}}\bigg)^{e_F}.
\end{eqnarray*}
To complete the proof, it suffices to show that $Y_F$ is concentrated around
its mean. For that we use Chebyshev's inequality (see e.g.~\cite{AlonSpencer}).
Thus it remains to show that $\var(Y_F) = o(\expec{Y_F}^2)$.

For $G \subset K_n$, let $I_G$ be the indicator random variable for the event
$\event{G \subseteq \TF_I}$.  We have
\begin{eqnarray*}
\var(Y_F) = \sum_{G, G'} \Cov(I_G, I_{G'}) = \sum_{G, G'} \expec{I_G I_{G'}} - \expec{I_G}\expec{I_{G'}},
\end{eqnarray*}
where the sum ranges over all copies $G, G'$ of $F$ in $K_n$.  We partition the
sum above to two sums and show that each is bounded by $o(\expec{Y_F}^2)$.
First, let $\sum_{G, G'}$ be the sum over all copies $G, G'$ of $F$ in $K_n$
such that $G$ and $G'$ share no vertex.  If $G$ and $G'$ share no vertex then
$G \cup G'$ is triangle-free. Hence, since the number of two vertex-disjoint copies of
$F$ in $K_n$
is asymptotically equal to the number of copies of $F$ in $K_n$ squared,
it follows from Theorem~\ref{thm:main:} that
\begin{eqnarray*}
\sum_{G, G'} \expec{I_G I_{G'}} - \expec{I_G}\expec{I_{G'}} = o\Bigg(
\bigg(\frac{v_F!}{\aut(F)} \binom{n}{v_F}
\bigg(\frac{\Phi(i\delta)}{\sqrt{n}}\bigg)^{e_F}\bigg)^2\Bigg) =
o(\expec{Y_F}^2).
\end{eqnarray*}
Next, we will make use of the following observation: if $G, G'$ are two copies
of $F$ in $K_n$ with $G \cap G'$ being isomorphic to $H$, then $\expec{I_G
I_{G'}} = O((n^{\epsilon-1/2})^{2e_F - e_H})$.  This is true since the
event $\event{G, G' \subseteq \TF_I}$ implies $\event{G \cup G' \subseteq
\B_{\le I}}$ and indeed, $\prob{G \cup G' \subseteq \B_{\le I}} =
O((n^{\epsilon-1/2})^{2e_F - e_H})$.
Let $\sum_{H}$ be the sum over all $H \subseteq F$ with $v_H \ge 1$.  Let
$\sum_{G \cap G' \equiv H}$ be the sum over all copies $G, G'$ of $F$ in $K_n$
that share at least $1$ vertex such that $G \cap G'$ is isomorphic to $H$.
Then by the observation above,
\begin{eqnarray*}
\sum_{H} \sum_{G \cap G' \equiv H} \Cov(I_G, I_{G'}) \le 
O(n^{2v_F - v_H}) \cdot (n^{\epsilon-1/2})^{2e_F - e_H},
\end{eqnarray*}
which, since $F$ is a fixed balanced graph with $e_F/v_F < 2$, is at most
$o(\expec{Y_{F}}^2)$ if $\epsilon \in (0, \epsilon_F)$ and $\epsilon_F$ is
sufficiently small.  This implies the desired bound on $\var(Y_{F})$.
\end{proof}

It remains to prove Theorem~\ref{thm:main:}. In the following section we state
two technical lemmas that will be used to prove Theorem~\ref{thm:main:}.  The
actual proof of Theorem~\ref{thm:main:} is given in Section~\ref{section:6}.
The rest of the paper will then be devoted for the proof of these technical
lemmas.

\section{Technical lemmas} \label{section:tech}
%
Here we state (and partly prove) two technical lemmas that will be used to
prove Theorem~\ref{thm:main:}.

We begin with some definitions.  For every edge $g \in K_n$ and for every $0
\le i \le I$, $j \in \{0,1,2\}$, define $\Lambda_j(g,i)$ as follows.  Let
$\Lambda_0(g,i)$ be the family of all sets $\{g_1, g_2\} \subseteq \TF_i$ such
that $\{g, g_1, g_2\}$ is a triangle.  Let $\Lambda_1(g,i)$ be the family of
all singletons $\{g_1\} \subseteq K_n \setminus \B_{\le i}$ such that there
exists $g_2 \in \TF_i$ for which $\{g, g_1, g_2\}$ is a triangle and it holds
that $\TF_i \cup \{g_1\}$ is triangle-free.  Let $\Lambda_2(g,i)$ be the family
of all sets $\{g_1,g_2\} \subseteq K_n \setminus \B_{\le i}$ such that $\{g,
g_1, g_2\}$ is a triangle and for which it holds that $\TF_i \cup \{g_j\}$ is
triangle-free for both $j \in \{1,2\}$.

\begin{definition}
For every $0 \le i \le I$, let
\begin{eqnarray*}
\gamma(i) &:=& \max\{\delta \Phi(i\delta) \phi(i\delta), \, \delta^2\phi(i\delta)^2 \}, \\
\Gamma(i) &:=&
 \left\{
\begin{array}{ll}
  n^{-30\epsilon} & \text{if } i = 0, \\
  \Gamma(i-1) \cdot (1 + 10 \gamma(i-1)) & \text{if } i \ge 1.
\end{array} \right.
\end{eqnarray*}
\end{definition}

Our first technical lemma tracks the cardinalities of $\Lambda_j(g,i)$.
%
\begin{lemma}
\label{lemma:edge}
%
Let $0 \le i < I$.  Suppose that given $\TF_i$, we have 
\begin{eqnarray*}
\forall g \in K_n. \,\,\,\,\,\,\, |\Lambda_0(g,i)| &\le& in^{5\epsilon},  \\
\forall g \in K_n. \,\,\,\,\,\,\, |\Lambda_1(g,i)| &\le& i \sqrt{n}, \\
\forall g \notin \B_{\le i}. \,\,\,\,\,\,\, |\Lambda_1(g,i)| &=& 2\sqrt{n} \Phi(i\delta) \phi(i\delta) \cdot (1 \pm \Gamma(i)),  \\
\forall g \notin \B_{\le i}. \,\,\,\,\,\,\, |\Lambda_2(g,i)| &=& n \phi(i\delta)^2 \cdot (1 \pm \Gamma(i)). 
\end{eqnarray*}
Then with probability $1-n^{-\omega(1)}$, 
\begin{eqnarray*}
\forall g \in K_n. \,\,\,\,\,\,\, |\Lambda_0(g,i+1)| &\le& (i+1)n^{5\epsilon}, \\
\forall g \in K_n. \,\,\,\,\,\,\, |\Lambda_1(g,i+1)| &\le& (i+1) \sqrt{n}, \\
\forall g \notin \B_{\le i+1}. \,\,\,\,\,\,\, |\Lambda_1(g,i+1)| &=& 2 \sqrt{n} \Phi((i+1)\delta) \phi((i+1)\delta) \cdot (1 \pm \Gamma(i+1)), \\
\forall g \notin \B_{\le i+1}. \,\,\,\,\,\,\, |\Lambda_2(g,i+1)| &=& n \phi((i+1)\delta)^2 \cdot (1 \pm \Gamma(i+1)). 
\end{eqnarray*}
\end{lemma}

The following fact will be used in several places in our proofs, either
explicitly or not, and its proof is given in Appendix~\ref{app:1}.
\begin{fact}
\label{fact:f1}
For all $0 \le i \le I$,
\begin{enumerate}
\item[(i)] $1 \ge \phi(i\delta) = \Omega(n^{-1.5\epsilon})$;
$\Phi(i\delta) \le \ln n$;
$i \ge 1 \implies \Phi(i\delta) = \Omega(n^{-\epsilon})$.
\item[(ii)] $\gamma(i) = o(1)$; $\gamma(i) = \Omega(n^{-5\epsilon})$;
$n^{-30\epsilon} \le \Gamma(i) \le n^{-10\epsilon}$.
\end{enumerate}
\end{fact}

\subsection{Proof of Lemma~\ref{lemma:edge}} \label{section:3}
%
Fix $0 \le i < I$ and assume that the precondition in Lemma~\ref{lemma:edge}
holds.  We prove that each of the consequences in Lemma~\ref{lemma:edge} hold
with probability $1-n^{-\omega(1)}$. Along the way we state a useful lemma
that, together with Lemma~\ref{lemma:edge}, will be used to prove
Theorem~\ref{thm:main:} in the next section.

For any $g \in K_n$, assuming $|\Lambda_0(g,i)| \le in^{5\epsilon}$, we
trivially have that $|\Lambda_0(g,i+1)| \le in^{5\epsilon} + \lambda_0(g)$,
where $\lambda_0(g)$ is the number of sets $\{g_1\} \in \Lambda_1(g,i)$ for
which it holds that $g_1 \in \B_{i+1}$, \emph{plus} the number of sets $\{g_1,
g_2\} \in \Lambda_2(g,i)$ for which it holds that $g_1, g_2 \in \B_{i+1}$.
Given the precondition in Lemma~\ref{lemma:edge}, the fact that
$|\Lambda_2(g,i)| \le n$, the definition of $\B_{i+1}$ and the fact that $i <
I$, it is clear that $\expec{\lambda_0(g)} = o(n^{5\epsilon})$. Hence, by
Chernoff's bound we get that with probability $1-n^{-\omega(1)}$, $\lambda_0(g)
\le n^{5\epsilon}$.  This implies that with probability $1-n^{-\omega(1)}$, for
all $g \in K_n$, $|\Lambda_0(g,i+1)| \le (i+1)n^{5\epsilon}$.  

Next, for any $g \in K_n$, assuming $|\Lambda_1(g,i)| \le i\sqrt{n}$, we
trivially have that $|\Lambda_1(g,i+1)| \le i\sqrt{n} + \lambda_1(g)$, where
$\lambda_1(g)$ is the number of sets $\{g_1, g_2\} \in \Lambda_2(g,i)$ for
which it holds that $g_1 \in \B_{i+1}$ and $g_2 \notin \B_{i+1}$.  By the fact
that $|\Lambda_2(g,i)| \le n$ and by the definition of $\B_{i+1}$, it is clear
that $\expec{\lambda_1(g)} = o(\sqrt{n})$. Hence, by Chernoff's bound we get
that with probability $1-n^{-\omega(1)}$, $\lambda_1(g) \le \sqrt{n}$.  This
implies that with probability $1-n^{-\omega(1)}$, for all $g \in K_n$,
$|\Lambda_1(g,i+1)| \le (i+1)\sqrt{n}$.  

\begin{remark} The only reason we are interested in maintaining the above upper
bound on the cardinality of $\Lambda_1(g,i)$ for all $g \in K_n$ and $i$, is
that we need this upper bound in order to maintain an upper bound on the
cardinality of $\Lambda_0(g,i)$ for all $g \in K_n$ and $i$ (as we did above).
We will not make any further use of the above upper bound on $\Lambda_1(g,i)$.
\end{remark}

Having dealt with the easy cases first, we now turn to deal with the two last,
more involved consequences in the lemma.

\subsubsection{Definitions and an observation}
\begin{definition}
[Redefinition of $\beta_{i+1}$]
Define $M := n^{20000\epsilon}$.  Let $\B_{i+1}^{\star}$ be a random set of edges,
formed by choosing every edge in $K_n \setminus \B_{\le i}$ with probability
$Mn^{-1/2}$.  For each $g \in \B_{i+1}^{\star}$, let $\beta_{i+1}(g)$ be distributed
uniformly at random in $[0,Mn^{-1/2}]$ and for each $g \in K_n \setminus
(\B_{\le i} \cup \B_{i+1}^{\star})$, let $\beta_{i+1}(g)$ be distributed uniformly at
random in $(Mn^{-1/2},1]$.
\end{definition}
Clearly, the above definition of $\beta_{i+1}$ is equivalent to the original
definition of $\beta_{i+1}$, given at Section~\ref{section:alternative}. Note
that the definition of $\B_{i+1}$ is not changed and that $\B_{i+1} \subseteq
\B_{i+1}^{\star}$. 

Let $\Lambda^{\star}_j(g, i)$ be the family of all $G \in \Lambda_j(g, i)$ such
that $G \subseteq \B_{i+1}^{\star}$. Let $\Lambda^{\star \star}_2(g, i)$ be the
family of all $G \in \Lambda_2(g, i)$ such that $|G \cap \B_{i+1}^{\star}| =
1$.  

\begin{definition} \label{def:Tgl}
Let $g \in K_n \setminus \B_{\le i}$, $l \in \N$.  We define inductively a
labeled rooted tree $T^{\star}_{g,l}$ of height $2l$.  The nodes at even
distance from the root will be labeled with edges from $K_n \setminus \B_{\le
i}$.  The nodes at odd distance from the root will be labeled with sets of $j
\in \{1,2\}$ edges from $K_n \setminus \B_{\le i}$.
\begin{itemize}
\item $T^{\star}_{g,1}$: 
\begin{itemize}
\item The root $v_0$ of $T^{\star}_{g,1}$ is labeled with the
edge $g$.
\item 
For every $G \in \Lambda^{\star}_1(g,i) \cup \Lambda^{\star}_2(g,i)$ do: set a
new node $u_1$, labeled $G$, as a child of $v_0$; furthermore, for each edge
$g_1 \in G$ set a new node $v_1$, labeled $g_1$, as a child of $u_1$.
\end{itemize}
\item 
$T^{\star}_{g,l}$, $l \ge 2$: We construct the tree $T^{\star}_{g,l}$ by adding
new nodes to $T^{\star}_{g,l-1}$ as follows.  Let
$(v_0,u_1,v_1,\ldots,u_{l-1},v_{l-1})$ be a directed path in
$T^{\star}_{g,l-1}$ from the root $v_0$ to a leaf $v_{l-1}$.  Let $g_j$ be the
label of $v_j$.
\begin{itemize}
\item
For every $G \in \Lambda^{\star}_2(g_{l-1},i)$ for which $g_{l-2} \notin G$ do:
set a new node $u_l$, labeled $G$, as a child of $v_{l-1}$; furthermore, for
each edge $g_l \in G$ set a new node $v_l$, labeled $g_l$, as a child of $u_l$.
\item 
For every $G \in \Lambda^{\star}_1(g_{l-1},i)$ for which $g_{l-2} \notin G$ and
$G \cup \{g_{l-1}, g_{l-2}\}$ isn't a triangle do: set a new node $u_l$,
labeled $G$, as a child of $v_{l-1}$; furthermore, for the edge $g_l \in G$ set
a new node $v_l$, labeled $g_l$, as a child of $u_l$.
\end{itemize}
\end{itemize}
Lastly, for $G \subset K_n \setminus \B_{\le i}$, define
$T^{\star}_{G,l} := \{T^{\star}_{g,l} : g \in G\}$. 
\end{definition}

Consider the tree $T^{\star}_{g,l}$.  Let $v$ be a node at even distance from the
root of $T^{\star}_{g,l}$. Let $f_0$ be the label of $v$.  We define the event that
$\emph{$v$ survives}$ as follows.  If $v$ is a leaf then $v$ survives by
definition.  Otherwise, $v$ survives if and only if for every child $u$,
labeled $G$, of $v$, the following holds: if $\beta_{i+1}(f) < \min\{
\beta_{i+1}(f_0), \delta n^{-1/2} \}$ for all $f \in G$ then $u$ has a child
that does not survive.  For $g \notin \B_{\le i}$, let $\A_{g,l}$ be the event
that the root of $T^{\star}_{g,l}$ survives. Let $\A_{G,l} := \bigcap_{g \in G}
\A_{g,l}$.  Given Definition~\ref{def:Tgl}, the following is an easy
observation.
\begin{proposition}
\label{prop:1}
Let $l \ge 1$ be an odd integer.
\begin{itemize}
\item
Conditioned on $\event{g \in \B_{i+1}, \TF_i \cup \{g\} \text{ is triangle-free}}$,
\begin{displaymath}
\A_{g,l} \implies \event{g \in \TF_{i+1}} \implies \A_{g,l+1}.
\end{displaymath}
\item 
Conditioned on $\event{g \notin \B_{\le i+1}, \TF_i \cup \{g\} \text{ is triangle-free}}$,
\begin{displaymath}
\A_{g,l} \implies \event{\TF_{i+1} \cup \{g\} \text{ is triangle-free}} \implies \A_{g,l+1}.
\end{displaymath}
\end{itemize}
\end{proposition}

\subsubsection{Proof of Lemma~\ref{lemma:edge}} \label{section:proof of main lemma}
%
Let $\E^{\star}$ be the event that the following properties hold:
\begin{itemize}
\item[P1] For every $g \notin \B_{\le i}$, 
\begin{eqnarray*}
|\Lambda^{\star}_1(g, i)| &=&
2M \Phi(i\delta) \phi(i\delta) \cdot (1 \pm (\Gamma(i) + o(\Gamma(i)\gamma(i))) ), \\
|\Lambda^{\star}_2(g, i)| &=&
M^2 \phi(i\delta)^2 \cdot (1 \pm (\Gamma(i) + o(\Gamma(i)\gamma(i)))), \\
|\Lambda^{\star \star}_2(g, i)| &=&
2M \sqrt{n} \phi(i\delta)^2 \cdot (1 \pm (\Gamma(i) + o(\Gamma(i)\gamma(i)))).
\end{eqnarray*}
\item[P2] For every three distinct vertices $w,x,y$, if $\{w, x\}, \{x, y\} \notin \B_{\le i}$:
\begin{itemize}
\item The number of vertices $z$ such that $\{w, z\}, \{y,
z\} \in \TF_i$ and $\{x, z\} \in \B_{i+1}^{\star}$ is at
most $(\ln n)^2$.
\item The number of vertices $z$ such that $\{w, z\} \in
\TF_i$ and $\{x, z\}, \{y, z\} \in \B_{i+1}^{\star}$ is at
most $(\ln n)^2$.
\item The number of vertices $z$ such that $\{w, z\} \in \TF_i$, $\{x, z\} \notin \B_{\le i}$ and
$\{y, z\} \in \B_{i+1}^{\star}$ is at most $M^2$.
\end{itemize}
\item[P3] For every two distinct vertices $x, y$, the number of vertices $z$ such that 
$\{x,z\}, \{y,z\} \in \B_{i+1}^{\star}$ is at most $2M^2$.
\item[P4] For every vertex $x$, the number of edges $\{x, y\} \in \B_{i+1}^{\star}$ is at most
$2M\sqrt{n}$.
\end{itemize}

Fix once and for the rest of the paper $L \in \{40, 41\}$.  
The following is our second technical lemma, which
is proved in Sections~\ref{section:4}--\ref{section:5}.
\begin{lemma}
\label{lemma:toprove:2}
\mbox{}
\begin{itemize}
\item $\prob{\E^{\star}} = 1 - n^{-\omega(1)}$.
\item Let $F \subset K_n \setminus \B_{\le i}$ be a triangle-free graph of size
$O(1)$ such that $\TF_i \cup F$ is triangle-free.  Assume $\B_{i+1}^\star$ was
chosen and condition on the event that $\E^{\star}$ holds.  Also condition on
the event that $a_1$ edges of $F$ are in $B_{i+1}$ and that $a_2$ edges of $F$
are not in $B_{i+1}$ (so that $|F| = a_1 + a_2$). Then
\begin{eqnarray*}
\prob{\A_{F,L}} = 
\bigg(\frac{ \Phi((i+1)\delta) - \Phi(i\delta)} {\phi(i\delta) \, \delta}\bigg)^{a_1}
\bigg(\frac{\phi((i+1)\delta)}{\phi(i\delta)}\bigg)^{a_2} 
\cdot  (1 \pm 4\Gamma(i) \gamma(i))^{a_1+a_2}.
\end{eqnarray*}
\end{itemize}
\end{lemma}
\begin{corollary}
\label{cor:1}
Suppose the settings and assumptions in the second item in
Lemma~\ref{lemma:toprove:2} hold.  Further assume that $g \notin \B_{\le i}
\cup F$ and condition on $\event{g \notin B_{i+1}}$.  Then
\begin{eqnarray*}
\prob{\A_{F,L}} = 
\bigg(\frac{ \Phi((i+1)\delta) - \Phi(i\delta)} {\phi(i\delta) \, \delta}\bigg)^{a_1}
\bigg(\frac{\phi((i+1)\delta)}{\phi(i\delta)}\bigg)^{a_2} 
\cdot  (1 \pm 4.01\Gamma(i) \gamma(i))^{a_1+a_2}.
\end{eqnarray*}
\end{corollary}
\begin{proof}
Suppose the settings and assumptions in the second item in
Lemma~\ref{lemma:toprove:2} hold and let $g \notin \B_{\le i} \cup F$.  Without
conditioning on $\event{g \notin \B_{i+1}}$, the corollary follows trivially
from Lemma~\ref{lemma:toprove:2}, only with the constant $4.01$ being replaced
by $4$.  Now note that we have $\prob{g \notin \B_{i+1}} \ge 1 - \delta M^{-1}
\ge 1 - o(\Gamma(i)\gamma(i))$, where the second inequality is by
Fact~\ref{fact:f1}.  This gives the corollary, since given $\E^{\star}$,
$\prob{\A_{F,L}} = 1-o(1)$.  (Indeed, $\A_{F,L}$ is implied by the event that
for all $f \in F$ and for all $G \in \Lambda^{\star}_j(f,i)$, $j \in \{1,2\}$,
there is an edge $g \in G$ which is not in $\B_{i+1}$. Given $\E^{\star}$ this
occurs with probability $1-o(1)$.)
\end{proof}

For the rest of the section we assume that we have already made the random
choices that determine the set $\B_{i+1}^{\star}$. We also assume that
$\E^{\star}$ holds and keep in mind the fact that this event holds with
probability $1-n^{-\omega(1)}$. We further fix for the rest of the section an
edge $g \notin \B_{\le i}$ and condition on the event $\event{g \notin
\B_{i+1}}$.  We estimate the cardinalities of $\Lambda_j(g,i+1)$ for $j \in
\{1,2\}$.

We define random variables that will be used to estimate the cardinality of
$\Lambda_j(g,i+1)$ for $j \in \{1,2\}$.  Let $\lambda_1(g, l)$ be the number of
sets $\{g_1\} \in \Lambda_1(g,i)$ for which it holds that $g_1 \notin \B_{i+1}$
and $\A_{g_1,l}$ occurs, \emph{plus} the number of sets $\{g_1, g_2\} \in
\Lambda_2^{\star}(g,i) \cup \Lambda_2^{\star \star}(g,i)$ for which it holds that $g_1 \in
\B_{i+1}$, $g_2 \notin \B_{i+1}$, and $\A_{g_1,l} \cap \A_{g_2,l}$ occurs.  Let
$\lambda_2(g, l)$ be the number of sets $\{g_1, g_2\} \in \Lambda_2(g,i)$ for
which it holds that $g_1, g_2 \notin \B_{i+1}$ and $\A_{g_1,l} \cap \A_{g_2,l}$
occurs. 

By definition of $\Lambda_j(g,i+1)$ and by Proposition~\ref{prop:1} we have for
odd $l \ge 1$,
\begin{eqnarray*}
\lambda_1(g,l)\,\,  \le&  |\Lambda_1(g,i+1)| & \le\,\,  \lambda_1(g,l+1), \\ 
\lambda_2(g,l)\,\,  \le&  |\Lambda_2(g,i+1)| & \le\,\,  \lambda_2(g,l+1). 
\end{eqnarray*}
Note that since $g \notin \B_{\le i}$, we have for all $F \in \Lambda_1(g,i)
\cup \Lambda_2(g,i)$ that $\TF_i \cup F$ is triangle-free.  Using this fact, we
can use Corollary~\ref{cor:1} together with the precondition in the lemma,
Fact~\ref{fact:f1} and the fact that $\prob{f \notin \B_{i+1}} \ge 1 - \delta
M^{-1}$ to verify that
\begin{eqnarray*}
\expec{\lambda_1(g,L)} &=&
2 \sqrt{n} \Phi((i+1)\delta) \phi((i+1)\delta) 
           \cdot (1 \pm (\Gamma(i) + 9\Gamma(i)\gamma(i))), \\
\expec{\lambda_2(g,L)} &=& 
n \phi((i+1)\delta)^2 
           \cdot (1 \pm (\Gamma(i) + 9\Gamma(i)\gamma(i))).
\end{eqnarray*}
We complete the proof by giving concentration results for $\lambda_j(g,L)$, $j
\in \{1,2\}$.  The required bound on the cardinality of $\Lambda_j(g,i+1)$ for
all $g \notin \B_{\le i+1}$ will then follow from these concentration results,
together with a union bound argument.

{\bf Concentration of $\lambda_1(g,L)$:}
Let $S_1$ be the set of edges which is the union of the sets in
$\Lambda_1(g,i)$, $\Lambda^{\star}_2(g,i)$ and $\Lambda_2^{\star \star}(g,i)$.
Let $S_2$ be the set of all nodes in the trees $T_{f,L}$, $f \in S_1$, where
$T_{f,L}$ is defined to be the tree that is obtained as follows: cut off from
$T_{f,L}^{\star}$ every subtree that is rooted at a node having a child that is
labeled $g$.  Let $S_3 \supset S_1$ be the set of edges that are labels of
nodes in~$S_2$.  By the precondition in Lemma~\ref{lemma:edge} and
$\E^{\star}$, we have that $|S_1| \le M^2n^{1/2}$ and that every tree $T_{f,L}$
has at most $O(M^{2L}) \le n^{1/1000}$ nodes. Therefore, $|S_3| \le |S_2| \le
M^2n^{1/2+1/1000}$.  Observe that since we condition on $\event{g \notin
\B_{i+1}}$, we have that for $f \in S_1$, $\A_{f,L}$ depends only on the
birthtimes of the edges that are labels in $T_{f,L}$.  Hence, since $S_1
\subseteq S_3$ we have that $\lambda_1(g,L)$ is determined by the birthtimes of
the edges in $S_3$.  We argue below that every edge in $S_3$ appears as a label
in at most $n^{1/1000}$ trees $T_{f,L}$, $f \in S_1$.  This implies that
changing the birthtime of a single edge in $S_3$ can change $\lambda_1(g,L)$ by
at most $n^{1/1000}$. It will then follow from Azuma's inequality, the bound
above on the number of edges in $S_3$, the bound on $\expec{\lambda_1(g,L)}$
and Fact~\ref{fact:f1} that, as needed, with probability $1-n^{-\omega(1)}$,
\begin{eqnarray*}
\lambda_1(g,L) = 2 \sqrt{n} \Phi((i+1)\delta) \phi((i+1)\delta) \cdot (1 \pm
\Gamma(i+1)). 
\end{eqnarray*}
We argue that every edge in $S_3$ appears as a label in at most $n^{1/1000}$
trees $T_{f,L}$, $f \in S_1$. For $g', g'' \in S_3$, say that $g'$
\emph{affects} (resp. \emph{directly-affects}) $g''$ if there is a tree
$T_{f,L}$, $f \in S_1$, with a path (resp. path of length $0$ or $2$) leading
from a node labeled $g''$ to a node labeled $g'$.  It is enough to show that
every edge in $S_3$ affects at most $n^{1/1000}$ edges in $S_1$.

Fix $g' \in S_3$. By Definition~\ref{def:Tgl}, if $g'' \in S_3 \setminus S_1$
then $g'' \in \B_{i+1}^{\star}$. Therefore, the number of edges $g'' \in S_3
\setminus S_1$ that $g'$ directly-affects is at most $|\Lambda_1^{\star}(g',i)|
+ |\Lambda_2^{\star}(g',i)| + 1 = O(M^2)$, where the upper bound is by
$\E^{\star}$.  If $g'$ shares no vertex with $g$ then it is clear that $g'$
directly-affects at most $5$ edges in $S_1$.  If $g'$ shares exactly $1$ vertex
with $g$ then one can verify that given the precondition in
Lemma~\ref{lemma:edge} and $\E^{\star}$ (specifically by~P1,~P2~and~P3), $g'$
directly-affects at most $O(M^2)$ edges in $S_1$.  This covers all possible
cases since $g' \ne g$ by definition of $S_3$.  We conclude that every edge in
$S_3$ directly-affects $O(M^2)$ edges in $S_3$. Since a path in $T_{f,L}$ has
length at most $2L$, and the edges that are labels along such a path are all in
$S_3$, we get that every edge in $S_3$ affects $O(M^{2L}) \le n^{1/1000}$ other
edges in $S_3$.  Since $S_1 \subseteq S_3$ we are done.

\begin{remark}
In the argument above, it was essential that we condition on $\event{g \notin
\B_{i+1}}$. Had we not done that, it would be the case that changing the
birthtime of $g$ would change $\lambda_1(g,L)$ potentially by at least
$|\Lambda_1(g,i)|$. This affect is too large, as it will render Azuma's
inequality useless in providing us with the concentration result we seek.
\end{remark}

{\bf Concentration of $\lambda_2(g,L)$:}
Let $S_1$ be the set of edges which is the union of the sets in
$\Lambda_2(g,i)$.  Let $S_2$ be the set of all nodes in the trees $T_{f,L}^{\star}$,
$f \in S_1$.  Let $S_3 \supset S_1$ be the set of edges that are labels of
nodes in $S_2$.  Trivially, $|S_1| \le 2n$. Also, by $\E^{\star}$ every tree
$T_{f,L}^{\star}$ has at most $O(M^{2L}) \le n^{1/1000}$ nodes. Therefore, $|S_3| \le
|S_2| \le 2n^{1+1/1000}$.  Observe that $\lambda_2(g,L)$ is determined by the
birthtimes of the edges in $S_3$.  We argue below that there is a set of at
most $M^2n^{1/2+1/1000}$ edges in $S_3$, each of which is a label in at most
$M^2n^{1/2+1/1000}$ trees $T_{f,L}^{\star}$, $f \in S_1$, and that every other edge
in $S_3$ is a label in at most $n^{1/1000}$ trees $T_{f,L}^{\star}$, $f \in S_1$.  It
will then follow from Azuma's inequality, the bound above on the number of
edges in $S_3$, the bound on $\expec{\lambda_2(g,L)}$ and Fact~\ref{fact:f1},
that as needed, with probability $1-n^{-\omega(1)}$,
\begin{eqnarray*}
\lambda_2(g,L) = n \phi((i+1)\delta)^2 \cdot (1 \pm \Gamma(i+1)). 
\end{eqnarray*}
Define \emph{affects} and \emph{directly-affects} exactly as above.  It is
enough to show that there is a set of at most $M^2n^{1/2+1/1000}$ edges in
$S_3$, each of which affects at most $M^2n^{1/2+1/1000}$ edges in $S_1$, and
that every other edge in $S_3$ affects at most $n^{1/1000}$ edges in $S_1$.

For a fixed edge $g' \in S_3$, we collect a few useful observations. First
assume that $g' \notin \B_{i+1}^{\star}$. Then by Definition~\ref{def:Tgl}, we
must have that $g' \in S_1$ and that $g'$ appears as a label only at the root
of $T_{g',L}^{\star}$. Therefore, if $g' \notin \B_{i+1}^{\star}$ then $g'$
affects (and directly-affects) only $g'$.  Next assume that $g' \in
B_{i+1}^{\star}$. By $\E^{\star}$ we have that $g'$ directly-affects at most
$|\Lambda_1^{\star}(g',i)| + |\Lambda_2^{\star}(g',i)| + 1 = O(M^2)$ edges $g''
\in S_3 \cap \B_{i+1}^{\star} \supset S_3 \setminus S_1$.  If $g'$ shares no
vertex with $g$ then $g'$ clearly directly-affects at most $5$ edges in $S_1$.
If $g'$ shares at least one vertex with $g$ then it follows from the
precondition in Lemma~\ref{lemma:edge} and $\E^{\star}$ (specifically by~P4)
that $g'$ directly-affects at most $M^2n^{1/2}$ edges in $S_1$.  
Lastly we note that for every $f \in S_1$ the following holds: every edge that
is a label in $T_{f,L}^{\star}$, except perhaps for $f$, is in $S_3 \cap
\B_{i+1}^{\star}$.

Say that $g'$ is a \emph{bad-edge} if $g'$ affects an edge $g'' \in S_3 \cap
\B_{i+1}^{\star}$ that shares at least one vertex with $g$.  From the
observations in the previous paragraph, it follows that if $g'$ is a bad-edge
then $g'$ affects at most $M^2n^{1/2} \cdot O(M^{2L}) \le M^2n^{1/2+1/1000}$
edges in $S_1$;  on the other hand, if $g'$ is not a bad-edge then $g'$ affects
at most $O(M^{2L}) \le n^{1/1000}$ edges in $S_1$. It thus remains to bound the
number of bad-edges.  By $\E^{\star}$ there are at most $M^2n^{1/2}$ edges $g''
\in \B_{i+1}^{\star}$ that share at least one vertex with $g$.  In addition, by
$\E^{\star}$, for every edge in $S_3$ there are at most $O(M^{2L}) \le
n^{1/1000}$ other edges that affect it.  Hence, there are at most
$M^2n^{1/2+1/1000}$ bad-edges.  With that we are done.

\section{Proof of Theorem~\ref{thm:main:}} \label{section:6}
%
Let $F \subset K_n$ be a triangle-free graph of size $O(1)$.  Say that the
triangle-free process \emph{well-behaves} if for every $0 \le i \le I$, the
precondition in Lemma~\ref{lemma:edge} holds.  Note that for $i=0$ the
precondition in Lemma~\ref{lemma:edge} holds trivially. Hence, by
Lemma~\ref{lemma:edge} and the union bound, the process well-behaves with
probability $1-n^{-\omega(1)}$.

For $0 \le i < I$, define
\begin{eqnarray*}
\varphi(i) := \frac{\Phi((i+1)\delta) - \Phi(i\delta)}{\delta}.
\end{eqnarray*}
In this section we will use $\alpha$ to denote a \emph{placement} $\event{f \in
\B_{i_f+1} : f \in F}$, where $0 \le i_f < I$ for all $f \in F$.  We will
show that for every fixed placement $\alpha$, 
\begin{eqnarray}
\label{eq:mt}
\probcnd{F \subseteq \TF_I, \text{process well-behaves}}{\alpha} \sim \prod_{f \in F} \varphi(i_f).
\end{eqnarray}
Note that for every placement $\alpha$, $\prob{\alpha} \sim
\big(\frac{\delta}{\sqrt{n}}\big)^{e_F}$. Taking $\sum_\alpha$ to be the sum over
all possible placements $\alpha$, it will then follow from~(\ref{eq:mt}) that
\begin{eqnarray*}
\prob{F \subseteq \TF_I, \text{process well-behaves}} &=& \sum_{\text{$\alpha$}} 
\prob{\alpha} \probcnd{F \subseteq \TF_I, \text{process well-behaves}}{\alpha} \\
&\sim& \bigg(\frac{\delta}{\sqrt{n}}\bigg)^{e_F} \sum_{\text{$\alpha$}}
\prod_{f \in F} \varphi(i_f) \\
&=& \bigg(\frac{\Phi(I\delta)}{\sqrt{n}}\bigg)^{e_F},
\end{eqnarray*}
where the validity of the last equality is by Claim~\ref{claim:lk} below.
Since the process well-behaves with probability $1-n^{-\omega(1)}$ and
$\Phi(I\delta) \to \infty$ as $n \to \infty$, it will then follow that, as needed, 
\begin{eqnarray*}
\prob{F \subseteq \TF_I} = n^{-\omega(1)} + 
\prob{F \subseteq \TF_I, \text{process well-behaves}} \sim
\bigg(\frac{\Phi(I\delta)}{\sqrt{n}}\bigg)^{e_F}.
\end{eqnarray*}
\begin{claim}
\label{claim:lk}
$\delta^{e_F} \sum_\alpha \prod_{f \in F} \varphi(i_f) = \Phi(I\delta)^{e_F}$.
\end{claim}
\begin{proof}
Let $\{Z_f : f \in F\}$ be a set of mutually independent $0/1$ random
variables, defined as follows. For every $f \in F$, choose uniformly at random
an index $0 \le i_f < I$ and let $f \in B_{i_f + 1}$. Then, let $Z_f = 1$ with
probability $\varphi(i_f)$. (We note that $\varphi(i) \in [0,1]$ for all $0 \le
i < I$; see Remark~\ref{remark:r2}.) In this context, the probability of a
placement $\alpha$ is $I^{-e_F}$. By definition we have
\begin{eqnarray*}
\prob{\forall f\in F. \,\,  Z_f = 1} = 
\sum_{\alpha} \prob{\alpha}\probcnd{\forall f\in F. \,\, Z_f = 1}{\alpha} 
= I^{-e_F} \sum_\alpha \prod_{f \in F} \varphi(i_f).
\end{eqnarray*}
On the other hand, by independence and symmetry we have for every fixed $g \in F$,
\begin{eqnarray*}
\prob{\forall f\in F. \,\, Z_f = 1}^{1/e_F} = \prob{Z_g = 1} 
= I^{-1} \sum_{0 \le i_g < I} \varphi(i_g) 
= (I\delta)^{-1} \Phi(I\delta).
\end{eqnarray*}
\end{proof}

It remains to prove~(\ref{eq:mt}). Fix a placement $\alpha$.  For $0 \le i <
I$, define $F_i := F \cap \B_{\le i}$.  For every $0 \le i \le I$, define the
events:
\begin{itemize}
\item[$\Q_1(i)$:] The precondition in Lemma~\ref{lemma:edge} holds for $i$. 
\item[$\Q_2(i)$:] $\TF_i \cup (F \setminus F_i)$ is triangle-free.
\item[$\Q_3(i)$:] $F_i \subseteq \TF_i$.
\end{itemize}
Let $\Q(i) := \Q_1(i) \cap \Q_2(i) \cap \Q_3(i)$.  Note that the event
$\bigcap_{0 \le i \le I} \Q(i)$ is exactly the event $\eventcnd{F \subseteq
\TF_I, \text{process well-behaves}}{\alpha}$. Therefore, it remains to estimate
the probability of $\bigcap_{0 \le i \le I} \Q(i)$.  Note that $\Q(0)$ holds
trivially. The next proposition gives an estimate on the probability that
$\Q(i+1)$ holds given $\Q(i)$.  Iterating on that proposition for all $0 \le i
< I$ gives~(\ref{eq:mt}).
\begin{proposition}
Let $0 \le i < I$ and assume $\Q(i)$ holds.  Then $\Q(i+1)$ holds with
probability
\begin{eqnarray*} 
\bigg(\frac{\Phi((i+1)\delta) - \Phi(i\delta)}{\phi(i\delta) \, \delta}\bigg)^{|F_{i+1} \setminus F_i|} 
\bigg(\frac{\phi((i+1)\delta)}{\phi(i\delta)}\bigg)^{|F \setminus F_{i+1}|} 
\cdot (1 \pm O(n^{-10\epsilon}))
\end{eqnarray*}
\end{proposition}
\begin{proof}
Assume that we are given an instance of $\TF_i$ and that $\Q(i)$ holds.
Consider the process as it creates $\TF_{i+1}$.  For the rest of the proof,
our context is the one given in Section~\ref{section:3}. 

Since $F \subset K_n$ is
triangle-free of size $O(1)$, we have by definition that $F \setminus F_i
\subset K_n \setminus \B_{\le i}$ is triangle-free of size $O(1)$.  By $\Q(i)$
we have that $\TF_i \cup (F \setminus F_i)$ is triangle-free.  Therefore, taking $a_1 =
|F_{i+1} \setminus F_i|$ and $a_2 = |F \setminus F_{i+1}|$, it follows from
Lemma~\ref{lemma:toprove:2} and Fact~\ref{fact:f1} that
\begin{eqnarray*}
\prob{\A_{F \setminus F_i,L}} = 
\bigg(\frac{ \Phi((i+1)\delta) - \Phi(i\delta)} {\phi(i\delta) \, \delta}\bigg)^{a_1}
\bigg(\frac{\phi((i+1)\delta)}{\phi(i\delta)}\bigg)^{a_2} 
\cdot  (1 \pm O(n^{-10\epsilon})).
\end{eqnarray*}
Let $F'$ be the set of all edges $f$ such that $F \cup \{f\}$ contains a
triangle and note that $|F'| = O(1)$.  Let $\E'$ be the event
that for every $f \in F'$, $\event{f \notin \TF_{i+1}}$. We have that
$\E'$ is implied by the event that for every $f \in F'$, $\event{f
\notin \B_{i+1}}$ occurs. Therefore, $\prob{\E'} \ge 1 - O(\delta n^{-1/2})
\ge 1-n^{-10\epsilon}$.  By Lemma~\ref{lemma:edge} we have $\prob{\Q_1(i+1)} =
1-n^{-\omega(1)}$.  Thus, since $\prob{\A_{F \setminus F_i,L}} = 1-o(1)$ (see the
proof of Corollary~\ref{cor:1}), it follows that
\begin{eqnarray*}
\prob{\A_{F\setminus F_i,L}, \E', \Q_1(i+1)} = 
\bigg(\frac{ \Phi((i+1)\delta) - \Phi(i\delta)} {\phi(i\delta) \, \delta}\bigg)^{a_1}
\bigg(\frac{\phi((i+1)\delta)}{\phi(i\delta)}\bigg)^{a_2} 
\cdot  (1 \pm O(n^{-10\epsilon})).
\end{eqnarray*}
All that is remained to observe is that the probability of $\event{\A_{F
\setminus F_i,L}, \E', \Q_1(i+1)}$ above is an estimation of the
probability of $\Q(i+1)$.  Indeed, it follows from Proposition~\ref{prop:1}
that if $L$ is odd (resp.  even) then $\event{\A_{F\setminus F_i,L}, \E',
\Q_1(i+1)}$ implies (resp. is implied by) $\Q(i+1)$.
\end{proof}

\section{Proof of Lemma~\ref{lemma:toprove:2}} \label{section:4}
%
Here we prove Lemma~\ref{lemma:toprove:2} modulo one lemma whose proof is given
in the next section.  Our context in this section and for the rest of the paper
is the one given in Section~\ref{section:3}, where the lemma was stated. That
is, we fix $0 \le i < I$ and assume the precondition in Lemma~\ref{lemma:edge}
holds.  

The first item in Lemma~\ref{lemma:toprove:2} follows from Chernoff's bound,
using Fact~\ref{fact:f1} and the precondition in Lemma~\ref{lemma:edge}.  Thus
it remains to prove the second item in the lemma. For the rest of the paper we
assume that $F \subset K_n \setminus \B_{\le i}$ is a triangle-free graph of
size $O(1)$ and that the preconditions in the second item of
Lemma~\ref{lemma:toprove:2} hold. That is, we assume that $\TF_i \cup F$ is
triangle-free, $\B_{i+1}^\star$ was chosen and $\E^\star$ holds. We also
condition on the event that $a_1$ edges of $F$ are in $\B_{i+1}$ and that $a_2$
edges of $F$ are not in $\B_{i+1}$. We remark that while we do have the set
$\B_{i+1}^{\star}$ at hand, we have not yet chosen the random function
$\beta_{i+1}$.

The basic idea of the proof is as follows.  We need to analyse the event
$\A_{F,L}$, and the definition of this event calls for a recursive analysis.
However, the fact that there could possibly be edges that are labels in more
than one node in $T_{F,L}^{\star}$ makes such a recursive analysis difficult.
As we insist on analysing $\A_{F,L}$ recursively, the following observation
comes to the rescue. Define $m := \floor{n^{100\epsilon}}\phi(i\delta)^{-1}$.
Redefine the birthtime function $\beta_{i+1}$ as follows. Let $\B_{i+1}^{*}$ be
a random set of edges formed by choosing every edge in $\B_{i+1}^{\star}$ with
probability $mM^{-1}$; for each $g \in \B_{i+1}^{*}$, let $\beta_{i+1}(g)$ be
distributed uniformly at random in $[0, mn^{-1/2}]$; for each $g \in
\B_{i+1}^{\star} \setminus \B_{i+1}^{*}$, let $\beta_{i+1}(g)$ be distributed
uniformly at random in $(mn^{-1/2}, Mn^{-1/2}]$ and for all other edges $g
\notin \B_{\le i}$, let $\beta_{i+1}(g)$ be distributed uniformly at random in
$(Mn^{-1/2}, 1]$. Let $\Lambda_j^{*}(g,i)$ be the set of all $G \in
\Lambda_j^{\star}(g,i)$ such that $G \subseteq \B_{i+1}^{*}$ and note that
$\B_{i+1} \subseteq \B_{i+1}^{*} \subseteq \B_{i+1}^{\star}$.  Let
$T_{f,L}^{*}$ be defined exactly as $T_{f,L}^{\star}$ only that now we use in
the definition $\Lambda_j^{*}(g,i)$ instead of $\Lambda_j^{\star}(g,i)$. Define
$T_{F,L}^{*} := \{T_{f,L}^{*} : f \in F\}$.
It turns out that with a sufficiently high probability, every edge that is a
label in $T_{F,L}^{*}$ is a label of exactly one node in $T_{F,L}^*$. Moreover,
in order to analyse the event $\A_{F,L}$, one only needs to consider the
birthtimes of the edges that are labels in $T_{F,L}^{*}$.  This will allow us
to analyse $\A_{F,L}$ recursively.

Let $\E^{*}$ be the event that for every $g \notin \B_{\le i}$,
\begin{eqnarray*}
|\Lambda^{*}_1(g, i)| &=& 2m \Phi(i\delta) \phi(i\delta) \cdot (1 \pm
1.01\Gamma(i)), \\
|\Lambda^{*}_2(g, i)| &=& m^2 \phi(i\delta)^2 \cdot (1 \pm 1.01\Gamma(i)).
\end{eqnarray*}
Let $\E^{*}_F$ be the following event: if $g$ is a label of some node at even
distance from the root of a tree in $T_{F,L}^{*}$, then $g$ is the label of no
other node at even distance from the root of a tree in $T_{F,L}^{*}$. 

The following two lemmas correspond to the basic idea outlined above, 
and clearly imply Lemma~\ref{lemma:toprove:2}.
\begin{lemma}
\label{lemma:E}
$\prob{\E^{*}, \E^{*}_F} = 1-O(M^{-1/10}) \ge 1-o(\Gamma(i)\gamma(i))$.
\end{lemma}
\begin{lemma}
\label{lemma}
Assume $\B_{i+1}^*$ was chosen and condition on $\E^* \cap \E^*_F$. Then
\begin{eqnarray*}
\prob{\A_{F,L}} =
\bigg(\frac{ \Phi((i+1)\delta) - \Phi(i\delta)} {\phi(i\delta) \, \delta}\bigg)^{a_1}
\bigg(\frac{\phi((i+1)\delta)}{\phi(i\delta)}\bigg)^{a_2} 
\cdot  (1 \pm 3.99\Gamma(i) \gamma(i))^{a_1+a_2}.
\end{eqnarray*}
\end{lemma}
The proof of Lemma~\ref{lemma:E} is given below. The proof of
Lemma~\ref{lemma} is given in the next section.

\subsection{Proof of Lemma~\ref{lemma:E}}
By Chernoff's bound we have $\prob{\E^{*}} = 1-n^{-\omega(1)}$.  Therefore, it is
enough to prove that $\prob{\E^{*}_F} \ge 1-M^{-1/10}$. We assume that $F$ is not
empty, otherwise the assertion is trivial. For brevity, set $\Lambda^{\star}(g,i)
:= \Lambda^{\star}_1(g,i) \cup \Lambda^{\star}_2(g,i)$ for all $g \in K_n$.  When
stating that two graphs share $a$ edges (or vertices), unless otherwise stated
this means that the two graphs share exactly $a$ edges (or vertices).

\begin{definition}
[bad-sequence]
\label{def:bad2}
Let $S = (G_1,G_2,\ldots,G_l)$ be a sequence of subgraphs of $K_n$ with $1 \le
l \le 2L$.  We say that $S$ is a \emph{bad-sequence} if the following
properties hold simultaneously.
\begin{itemize}
\item For every $j \in [l]$: $G_j \in \Lambda^{\star}(g,i)$ for some $g \in F \cup
\bigcup_{k<j} G_k$.
\item For every $j \in [l-1]$: $G_j$ shares $|G_j|$ vertices and $0$ edges with
$F \cup \bigcup_{k<j} G_k$. 
\item Either 
  \begin{itemize}
  \item $G_l$ shares $|G_l|+1$ vertices and at most $|G_l|-1$ edges with $F \cup
  \bigcup_{k<l} G_k$, or
  \item $G_l$ shares $|G_l|$ vertices and $0$ edges with $F \cup \bigcup_{k<l}
  G_k$.  In addition, there is an edge $\{x, y\} \in F \cup \bigcup_{k<l} G_k$ such that
  $G_l \in \Lambda^{\star}(\{x, y\}, i)$ and there is an edge in $G_l$, without loss of
  generality $\{x, z\}$, with the following property: there is an edge 
  $\{w, x\} \in F \cup \bigcup_{k<l} G_k$ with $w \ne y$ such that 
  $\{w, z\} \in \TF_i$.
  \end{itemize}
\end{itemize}
\end{definition}

Let $\E$ be the event that for every bad-sequence $S = (G_1, G_2, \ldots, G_l)$
there exists $j \in [l]$ such that $\event{G_j \nsubseteq \B_{i+1}^{*}}$. The next
two propositions imply the desired bound $\prob{\E^{*}_F} \ge 1 - M^{-1/10}$, as
they state that $\E$ implies $\E^{*}_F$ and $\prob{\E} \ge 1 - M^{-1/10}$.

\begin{proposition}
\label{prop:lkj}
$\E$ implies $\E^{*}_F$.
\end{proposition}
\begin{proof}

Assume $\E$ occurs. We have the following claim.

\begin{claim}
\label{claim:clm0}
Let $P = (v_0, u_1, v_1, \ldots, u_L, v_L)$ denote an arbitrary path in
$T_{F,L}^{*}$, starting with some root $v_0$ and ending with some leaf. Let $G_j$
be the label of node $u_j$ and let $g_j$ be the label of node $v_j$ (so that
$g_0 \in F$).  Then for every $j \in [L]$: $G_j$ shares $0$ edges with $F \cup
\bigcup_{k<j} G_k$.
\end{claim}
\begin{proof}
Suppose for the sake of contradiction that the claim is false, and fix the
minimal $l \in [L]$ for which $G_l$ shares at least one edge with $F \cup \bigcup_{k<l}
G_k$. Consider the sequence $S = (G_1, G_2, \ldots, G_l)$. We shall reach a
contradiction by showing that $S$ or some prefix of $S$ is a bad-sequence.

A key observation is this: for all $j \in [l-1]$, $G_j$ shares $|G_j|$ vertices
and $0$ edges with $F \cup \bigcup_{k<j} G_k$.  To see that the observation
holds, first note that the minimality of $l$ implies that for all $j \in
[l-1]$, $G_j$ shares $0 \le |G_j| - 1$ edges with $F \cup \bigcup_{k<j} G_k$.
In addition, trivially, for all $j \in [l-1]$, $G_j$ shares at least $|G_j|$
vertices with $F \cup \bigcup_{k<j} G_k$. These facts together with $\E$ now
imply that there cannot be $j \in [l-1]$ such that $G_j$ shares $|G_j| + 1$
vertices with $F \cup \bigcup_{k<j} G_k$. 

Suppose that $|G_l| = 2$.  By assumption we have that $G_l$ shares at least one
edge with $F \cup \bigcup_{k<l} G_k$, which also implies that $G_l$ shares
$|G_l| + 1$ vertices with $F \cup \bigcup_{k<l} G_k$.  Hence, by the key
observation above, in order to show that $S$ is a bad-sequence and reach a
contradiction, it remains to show that $G_l$ shares $1 = |G_l| - 1$ edge with
$F \cup \bigcup_{k<l} G_k$.
Suppose on the contrary that $G_l$ shares both of its $2$ edges with $F \cup
\bigcup_{k<l} G_k$.  Notice that since $F$ is triangle-free, this implies that
$l \ge 2$, so $g_{l-2}$ is well defined.  Write $g_{l-2} = \{x,y\}$ and
$g_{l-1} = \{x,z\}$ and note that $z \notin \{x, y\}$, $G_{l-1} \in
\Lambda^{\star}(g_{l-2},i)$ and $G_l \in \Lambda^{\star}(g_{l-1},i)$.  Now,
note that the edge in $G_l$ that is adjacent to $z$ must also be an edge in
$G_{l-1}$. This is true since otherwise, $G_{l-1}$ will share the vertex $z$
with $F \cup \bigcup_{k<l-1} G_k$, which is clearly not the case as by the key
observation above $G_{l-1}$ shares only vertices from $\{x, y\}$ with $F \cup
\bigcup_{k<l-1} G_k$.  The only possible edge to be adjacent in $G_l$ to $z$
and be in $G_{l-1}$ is the edge $\{y, z\}$. Hence we get that $y$ is a vertex
of $G_l$. Therefore, we conclude that $g_{l-2} \in G_l$.  But by the definition
of $T_{F,L}^{*}$, $g_{l-2} \notin G_l$. Thus, $G_l$ shares $1$ edge with $F
\cup \bigcup_{k<l} G_k$ as needed.

Next assume that $|G_l| = 1$. By assumption we have that $G_l$ shares its edge
with $F \cup \bigcup_{k<l} G_k$. Since $\TF_i \cup F$ is triangle-free, this
implies that $l \ge 2$ and so $g_{l-2}$ is well defined. Let $x, y, z$ be as
defined in the previous paragraph. Note that either $x$ or $z$ are vertices of
$G_l$.  First we claim that $z$ cannot be a vertex of $G_l$.  Indeed, if $z$
was a vertex of $G_l$ then by a similar argument to that in the previous
paragraph we get that $G_l$ must be the edge $\{y,z\}$.  But this implies that
$\{g_{l-2}, g_{l-1}, g_l\}$ is a triangle and thus contradicts the definition
of $T_{F,L}^{*}$.  Therefore, $x$ is a vertex of $G_l$. We next argue that
$(G_j)_{j=1}^{l-1}$ is a bad-sequence, and by that get a contradiction.
Note that $\{x,y\}$ is an edge in $F \cup \bigcup_{k<l-1} G_k$ such that
$G_{l-1} \in \Lambda^{\star}(\{x,y\},i)$ and that $\{x,z\}$ is an edge in $G_{l-1}$.
Let $\{w,x\}$ be the edge in $G_l$. By definition of $T_{F,L}^{*}$ we have that $w
\notin \{x,y,z\}$.  This implies, since we assume that $\{w,x\}$ is an edge in
$F \cup \bigcup_{k<l} G_k$, that $\{w,x\}$ is an edge in $F \cup
\bigcup_{k<l-1} G_k$.  Since $|G_l| = 1$ we have that $\{w,z\} \in \TF_i$.
With the key observation above it now follows by definition that
$(G_j)_{j=1}^{l-1}$ is a bad-sequence. 
\end{proof}

The next claim, when combined with Claim~\ref{claim:clm0}, implies the
proposition.  \begin{claim}
Fix $1 \le l \le L$ and let $u$ be a node at distance $2l-1$ from a root in
$T_{F,L}^{*}$.  Fix $1 \le l' \le l$ and let $u'$ be a different node at distance
$2l'-1$ from a root in $T_{F,L}^{*}$. Then the labels of $u$ and $u'$ share $0$ edges.
\end{claim} \begin{proof}
The proof is by induction on $l$.  For the base case $l=1$, let $u$ and $u'$ be
two distinct nodes at distance $1$ from the roots of $T_{F,L}^{*}$.  Let $G$
and $G'$ be the labels of $u$ and $u'$ respectively.  Assume for the sake of
contradiction that $G$ and $G'$ share at least one edge.  We claim that either
$(G)$ or $(G, G')$ is a bad-sequence thus reaching the desired contradiction.
To see that this indeed holds, note first that by Claim~\ref{claim:clm0}, $G$
shares $|G|$ vertices and $0$ edges with $F$ and $G'$ shares $|G'|$ vertices
and $0$ edges with $F$. Let $v$ and $v'$ be the parents of $u$ and $u'$
respectively. Since $G$ and $G'$ share at least one edge and $u \ne u'$, we
have that $v \ne v'$. Therefore $G$ and $G'$ share exactly $1$ edge.
Now, if $|G'| = 2$ it follows that $G'$ shares $|G'|+1$ vertices and $|G'|-1$
edges with $F \cup G$; this implies that $(G, G')$ is a bad-sequence.  Next,
assume $|G'| = 1$.  Let $\{x,y\} \in F$ and $\{w,x\} \in F$ be the labels of
$v$ and $v'$ respectively.  Let $z$ be the vertex of $G$ and $G'$ that is not
in $F$ so that $G$ and $G'$ share the edge $\{x,z\}$. Clearly $w \ne y$.  In
addition, since $|G'| = 1$ we have that $\{w, z\} \in \TF_i$.  It follows
that $(G)$ is a bad-sequence.

Fix $2 \le l \le L$ and assume the claim is valid for $l-1$. Let $u$ be a node
at distance $2l-1$ from a root in $T_{F,L}^{*}$. Fix $1 \le l' \le l$ and let
$u'$ be a different node at distance $2l'-1$ from a root in $T_{F,L}^{*}$.
Assume for the sake of contradiction that the label of $u$ shares at least one
edge with the label of $u'$.  Without loss of generality we further assume that
$l'$ is minimal in the following sense: the label of $u$ shares $0$ edges with
the label of every node at odd distance less than $2l' - 1$ from the root of
$T_{F,L}^{*}$.  By Claim~\ref{claim:clm0}, we may also assume that $u'$ is not
a node on the path from a root to $u$ in $T_{F,L}^{*}$.

Let $P$ be the unique path from a root to $u$ in $T_{F,L}^{*}$. Let $P'$ be the
longest unique path in $T_{F,L}^{*}$ that ends with $u'$ and which do not
contain a node from $P$.  Traverse the nodes along the path $P$ and then
traverse the nodes along the path $P'$, ending each traversal at the nodes $u$
and $u'$ respectively.  Let $(u_1, u_2, \ldots, u_s)$ be the nodes so traversed
that are at odd distances from the roots of the forest, in order of their
traversal. By construction, $u_l = u$ and $u_s = u'$. We note that $2 \le s \le
2L$.  Let $G_j$ be the label of node $u_j$ and set $S_1 = (G_1, G_2, \ldots,
G_s)$.  Let $S_2 = (G_1, G_2, \ldots, G_{l-1}, G_{l+1}, G_{l+2}, \ldots,
G_{s-1}, G_l)$. In words, $S_2$ is obtained from $S_1$ by first removing $G_l$
and $G_s$ and then concatenating $G_l$ to the end of the new sequence.  Note
that $G_l$ and $G_s$ are the labels of $u$ and $u'$ respectively and that by
assumption $G_l$ and $G_s$ share at least one edge.  We show below that either
$S_1$ or $S_2$ is a bad-sequence and by that reach the desired contradiction.

Assume that $|G_s| = 2$. We show that $S_1$ is a bad-sequence.  By
Claim~\ref{claim:clm0}, the minimality of $l'$ and the induction hypothesis we
have that for every $j \in [s-1]$, $G_j$ shares $0$ edges with $F \cup
\bigcup_{k<j} G_k$. Therefore, by $\E$ we also have that for every $j \in
[s-1]$, $G_j$ shares $|G_j|$ vertices with $F \cup \bigcup_{k<j} G_k$.  Let
$v_l$ be the parent of $u_l$ and $v_s$ the parent of $u_s$.  Let $g_l$ be the
label of $v_l$ and $g_s$ the label of $v_s$.  Since $u_l \ne u_s$ and yet $G_l$
and $G_s$ share at least one edge, we get that $v_l \ne v_s$. This implies by
Claim~\ref{claim:clm0} and the induction hypothesis that $g_l \ne g_s$. This,
in turn, implies that $G_l$ shares \emph{exactly} $1$ edge with $G_s$. In what
follows we show that $G_s$ shares $0$ edges with $F \cup \bigcup_{k < l, l < k
< s} G_k$.  This will give us that $G_s$ shares $|G_s| + 1$ vertices and $1 =
|G_s| - 1$ edges with $F \cup \bigcup_{k < s} G_k$, which given the above
implies that $S_1$ is a bad-sequence.
The fact that $G_s$ shares $0$ edges with $F \cup \bigcup_{l < k < s} G_k$
follows from Claim~\ref{claim:clm0}. We claim that $G_s$ shares $0$ edges with
$\bigcup_{k < l} G_k$. Indeed, if $G_s$ does share at least one edge with
$\bigcup_{k < l} G_k$, then since $G_s$ also shares at least one edge with
$G_l$, we get that $G_l$ shares $|G_l| + 1$ vertices with $F \cup \bigcup_{k <
l} G_k$. But since $l \in [s-1]$, we've ruled out that possibility above. 

Assume that $|G_s| = 1$. We show that $S_2$ is a bad-sequence.  For brevity,
rewrite $S_2 = (F_1, F_2, \ldots, F_{s-1})$ and note that $F_{s-1}$ is the
label of $u_l$.  By Claim~\ref{claim:clm0}, the minimality of $l'$ and the
induction hypothesis we have that for every $j \in [s-1]$, $F_j$ shares $0$
edges with $F \cup \bigcup_{k<j} F_k$. Therefore, by $\E$ we also have that for
every $j \in [s-1]$, $F_j$ shares $|F_j|$ vertices with $F \cup \bigcup_{k<j}
F_k$.
Define $g_l, g_s$ as in the previous paragraph and note that for the same
reasons as above we have that $g_l \ne g_s$. Also note that $g_l, g_s \in F
\cup \bigcup_{k<s-1} F_k$. Write $g_l = \{x, y\}$ and let $z$ be the vertex of
$G_l$ that is not in $\{x, y\}$.  Assume without loss of generality that $G_l$
and $G_s$ share the edge $\{x, z\}$.  Since $z$ is not a vertex of $F \cup
\bigcup_{k<s-1}$, we get that $x$ is a vertex in $g_s$. Write $g_s = \{w,x\}$
and note that $w \ne y$.  Lastly, since $|G_s| = 1$ we have that $\{w, z\} \in
\TF_i$.  Therefore, by definition, $S_2$ is a bad-sequence.
\end{proof} 
With that we complete the proof of the proposition.
\end{proof}

\def\Seq{\textrm{Seq}}
\def\Seqq#1#2{\textrm{Seq}_{#2}(#1)}

\begin{proposition}
\label{eq:pe3}
$\prob{\E} \ge 1 - M^{-1/10}$.
\end{proposition}
\begin{proof}
For a bad-sequence $S = (G_1, G_2, \ldots, G_l)$, write $\event{S \subseteq
\B_{i+1}^{*}}$ for the event that for all $j \in [l]$, $\event{G_j \subseteq
\B_{i+1}^{*}}$. Let $Z$ be the random variable that counts the number of
bad-sequences $S$ for which $\event{S \subseteq \B_{i+1}^{*}}$.  It suffices to
show that $\expec{Z} \le M^{-1/10}$.

For $l \in [2L]$, $0 \le c < l$, let $\Seqq{l,c}{1}$ denote the set of all
bad-sequences $S = (G_1, G_2, \ldots, G_l)$ with $c = |\{j : |G_j| = 1, j <
l\}|$ such that $G_l$ shares $|G_l|+1$ vertices and at most $|G_l|-1$ edges
with $F \cup \bigcup_{k<l} G_k$.  For $l \in [2L]$, $0 \le c < l$, let
$\Seqq{l,c}{2}$ denote the set of all bad-sequences $S = (G_1, G_2, \ldots,
G_l)$ with $c = |\{j : |G_j| = 1, j < l\}|$ that are not in $\Seqq{l,c}{1}$.
Then
\begin{eqnarray}
\label{eq:sum2}
\expec{Z} &=& 
  \sum_{l \in [2L]} 
  \sum_{0 \le c < l}
  \sum_{j \in \{1, 2\}} 
  \sum_{S \in \Seqq{l,c}{j}} \prob{S \subseteq \B_{i+1}^{*}}.
\end{eqnarray}
Below we show that 
\begin{eqnarray}
\label{eq:pil1}
\forall l \in [2L], \,\, 0 \le c < l. \,\,\,\,\,
\sum_{S \in \Seqq{l,c}{1}} \prob{S \subseteq \B_{i+1}^{*}} &\le& M^{-1/9}, \\
\label{eq:pil2}
\forall l \in [2L], \,\, 0 \le c < l. \,\,\,\,\,
\sum_{S \in \Seqq{l,c}{2}} \prob{S \subseteq \B_{i+1}^{*}} &\le& M^{-1/9}.
\end{eqnarray}
From~(\ref{eq:sum2}),~(\ref{eq:pil1})~and~(\ref{eq:pil2}) and since $L = O(1)$,
we get that $\expec{Z} \le M^{-1/10}$ as required.  

We prove~(\ref{eq:pil1}). Fix $l \in [2L]$, $0 \le c < l$. We first count the
number of sequences $S = (G_1, G_2, \ldots, G_l)$ in $\Seqq{l,c}{1}$. To do so,
we construct such a sequence iteratively. First, we choose the cardinalities of
the first $l-1$ subgraphs in $S$. Note that there are $\binom{l-1}{c} = O(1)$
possible choices for the cardinalities.  Suppose we have already chosen the
first $j-1$ subgraphs in $S$ for some $j < l$. Given that, we count the number
of choices for $G_j$ assuming $j \ge 1$.  There are $O(1)$ possible choices for
an edge $g \in F \cup \bigcup_{k<j} G_k$ for which $G_j \in \Lambda^{\star}(g,
i)$.  Given $g$: if $|G_j|$ is to be of size $1$ then there are at most
$\Lambda^{\star}_1(g, i)$ choices for $G_j$ and if $|G_j|$ is to be of size $2$
then there are at most $\Lambda^{\star}_2(g, i)$ choices for $G_j$.  Given that
we have already chosen the first $l-1$ subgraphs in $S$, the number of choices
for $G_l$ is at most $O(1)$, since the vertices of $G_l$ are all in $F \cup
\bigcup_{k<l} G_k$.  Therefore, by $\E^{\star}$
the number of sequences in $\Seqq{l,c}{1}$ is at most
\begin{eqnarray*}
O(1) \cdot \big(M^2\phi(i\delta)^2\big)^{l-1-c} \cdot \big(M\Phi(i\delta)\phi(i\delta)\big)^{c}.
\end{eqnarray*}
Even if we condition on the event that $a_1$ edges of $F$ are in $\B_{i+1}$ and
$a_2$ edges of $F$ are not in $\B_{i+1}$, we get that the probability of
$\event{S \subseteq \B_{i+1}^{*}}$ for $S \in \Seqq{l,c}{1}$ is at most 
\begin{eqnarray}
\label{eq:ka}
\bigg(\frac{m^2}{M^2}\bigg)^{l-1-c} \cdot \bigg(\frac{m}{M}\bigg)^{c} \cdot \frac{m}{M}.
\end{eqnarray}
Hence, 
\begin{eqnarray*}
\sum_{S \in \Seqq{l,c}{1}} \prob{S \subseteq \B_{i+1}^{*}} &\le&
    O(1) \cdot
    \big(m^2\phi(i\delta)^2\big)^{l-1-c} \cdot 
    \big(m\Phi(i\delta)\phi(i\delta)\big)^{c} \cdot \frac{m}{M} \\
&\le& 
    O(1) \cdot m^{2l-2-2c} \cdot (m\ln n)^c \cdot \frac{m}{M} \\
&\le& M^{-1/9},
\end{eqnarray*}
where the second inequality follows from Fact~\ref{fact:f1} and the last
inequality follows from the definition of $L, m$ and $M$.  This gives us the
validity of~(\ref{eq:pil1}).

It remains to prove~(\ref{eq:pil2}).  Fix $l \in [2L]$, $0 \le c < l$. As
before, we first count the number of sequences $S = (G_1, G_2, \ldots, G_l)$ in
$\Seqq{l,c}{2}$ and we do it by constructing such a sequence iteratively.  The
number of choices for the first $l-1$ subgraphs in $S$ is exactly as in the
previous case.  Suppose we have already chosen the first $l-1$ subgraphs in
$S$.  We claim that the number of choices for $G_l$ is at most $O((\ln n)^2)$.
Indeed, there are $O(1)$ choices for an edge $\{x, y\} \in F \cup \bigcup_{k<l}
G_k$ such that $G_l \in \Lambda^{\star}(\{x,y\},i)$. Given $\{x,y\}$, there are
at most $O(1)$ choices for an edge $\{w, x\} \in F \cup \bigcup_{k<l} G_k$ such
that $w \ne y$. Furthermore, given $\{x,y\}$ and $\{w,x\}$, by $\E^{\star}$
(specifically by P2) there are at most $2(\ln n)^2$ choices for $G_l \in
\Lambda^{\star}(\{x,y\},i)$ which has a vertex $z$ that is not a vertex of $F
\cup \bigcup_{k<l} G_k$ and such that $\{x, z\} \in G_l$ and $\{w,z\} \in
\TF_i$. Therefore, by $\E^{\star}$ the number of sequences in $\Seqq{l,c}{2}$
is at most
\begin{eqnarray*}
O(1) \cdot \big(M^2\phi(i\delta)^2\big)^{l-1-c} \cdot \big(M\Phi(i\delta)\phi(i\delta)\big)^{c} 
   \cdot (\ln n)^2.
\end{eqnarray*}
Even if we condition on the event that $a_1$ edges of $F$ are in $\B_{i+1}$ and
$a_2$ edges of $F$ are not in $\B_{i+1}$, we get that the probability of
$\event{S \subseteq \B_{i+1}^{*}}$ for $S \in \Seqq{l,c}{2}$ is at
most as given in~(\ref{eq:ka}).  Therefore,
\begin{eqnarray*}
\sum_{S \in \Seqq{l,c}{2}} \prob{S \subseteq \B_{i+1}^{*}} &\le&
    O(1) \cdot
    \big(m^2\phi(i\delta)^2\big)^{l-1-c} \cdot 
    \big(m\Phi(i\delta)\phi(i\delta)\big)^{c} \cdot \frac{m}{M} \cdot (\ln n)^2\\
&\le& 
    O(1) \cdot m^{2l-2-2c} \cdot (m\ln n)^c \cdot \frac{m}{M} \cdot (\ln n)^2\\
&\le& M^{-1/9},
\end{eqnarray*}
where as before, the second inequality follows from Fact~\ref{fact:f1} and the
last inequality follows from the definition of $L, m$ and $M$.  This gives us
the validity of~(\ref{eq:pil2}). With that we complete the proof.
\end{proof}

\section{Proof of Lemma~\ref{lemma}} \label{section:5}
%
Assume that $\B_{i+1}^*$ was chosen and condition on $\E^{*} \cap \E^{*}_F$.
Fix an edge $f \in F$. Note that we either condition on the event that $f$ is
in $\B_{i+1}$ or not. For simplicity of presentation, we do not choose right
now which of these two options hold. The exact choice will be made implicitly
below, whenever we condition on an event which is concerned with the birthtime
$\beta_{i+1}(f)$.

Some remarks regarding $T_{f,L}^{*}$ follow.  The event
$\E^{*}_{F}$ says that every label of some node in $T_{f,L}^{*}$ is a label of
exactly one node in $T_{f,L}^*$.  Therefore, we shall refer from now on to the
nodes of $T_{f,L}^{*}$ \emph{by their labels}.  The event $\E^{*}$ implies,
using the definition of $T_{f,L}^{*}$ and Fact~\ref{fact:f1}, that for every
non-leaf node $g$ at even distance from the root of $T_{f,L}^{*}$,
\begin{eqnarray}
\label{eq:child1} \text{Number of children of $g$ that are of size $1$} &=& 2m
\Phi(i\delta) \phi(i\delta) (1 \pm 1.02\Gamma(i)), \\
\label{eq:child2} \text{Number of children of $g$ that are of size $2$} &=& m^2
\phi(i\delta)^2 (1 \pm 1.02\Gamma(i)).
\end{eqnarray}

We need to define the following two additional rooted trees.
\begin{definition}[$T_\infty, T_l$] \mbox{} \begin{itemize}
\item Let $T_\infty$ be an infinite rooted tree, defined as follows.  Every
node $g$ at even distance from the root has two sets of children.  One set
consists of children which are singletons and the other set consists of
children which are sets of size $2$.  Every node $G$ at odd distance from the
root of $T_\infty$, which is a set of size $|G| \in \{1, 2\}$, has exactly
$|G|$ children. Lastly, for every node $g$ at even distance from the root:
\begin{eqnarray*}
\text{Number of children of $g$ that are of size $1$} &=&
\ceil{2m \Phi(i\delta) \phi(i\delta)}, \\
\text{Number of children of $g$ that are of size $2$} &=&
m^2 \phi(i\delta)^2.
\end{eqnarray*}
\item Let $0 \le l \le L$. Define $T_l$ to be the tree that is obtained by
cutting from $T_\infty$ every subtree that is rooted at a
node whose distance from the root of $T_\infty$ is larger than $2l$.
\end{itemize}
\end{definition}
\begin{remark}
Note that $m^2\phi(i\delta)^2$ is an integer.  It would be convenient to assume
from now on that $2m\Phi(i\delta)\phi(i\delta)$ is also an integer. Hence, for
example, the number of children of the root of $T_\infty$ that are of size $1$
is exactly $2m\Phi(i\delta)\phi(i\delta)$. We explain in Section~\ref{sec:last}
how to modify our proof for the case where $2m\Phi(i\delta)\phi(i\delta)$ is
not an integer. 
\end{remark}

We continue with some more setup.
Note that for every node $g \ne f$ at even distance from the root of
$T_{f,L}^{*}$, $\beta_{i+1}(g)$ is distributed uniformly at random in the
interval $[0, mn^{-1/2}]$.  We extend the definition of $\beta_{i+1}$ so that
in addition, for every node $g$ at even distance from the root of $T_\infty$
(and hence from the root of $T_L$), the \emph{birthtime} $\beta_{i+1}(g)$ is
distributed uniformly at random in the interval $[0, mn^{-1/2}]$.

Let $T \in \{T_{f,L}^{*}, T_L, T_\infty\}$.  Let $g_0$ be a node at even
distance from the root of $T$.  We define the event that $\emph{$g_0$
survives}$ as follows.  If $g_0$ is a leaf (so that $T \ne T_\infty$) then
$g_0$ survives by definition.  Otherwise, $g_0$ survives if and only if for
every child $G$ of $g_0$, the following holds: if $\beta_{i+1}(g) < \min\{
\beta_{i+1}(g_0), \delta n^{-1/2} \}$ for all children $g$ of $G$, then $G$ has
a child that does not survive.

For a node $g$ at height $2l$ in $T_{f,L}^{*}$, let $p_{g,l}(x)$ be the probability
that $g$ survives under the assumption that $\beta_{i+1}(g) = xn^{-1/2}$.  Let
$p_l(x)$ be the probability, at the limit as $n \to \infty$, that the root of
$T_l$ survives under the assumption that $\beta_{i+1}(g) = xn^{-1/2}$, where
$g$ here denotes the root of $T_l$.  Let $p(x)$ be the probability, at the
limit as $n \to \infty$, that the root of $T_\infty$ survives under the
assumption that $\beta_{i+1}(g) = xn^{-1/2}$, where $g$ here denotes the root
of $T_\infty$.
One can show that $p_{g,l}(x), p_l(x)$ and $p(x)$ are all continuous and
bounded in the interval $[0,\delta]$. Hence, we can define the following
functions on the interval $[0,\delta]$:
\begin{eqnarray*}
P_{g,l}(x) := \int_0^x p_{g,l}(y) dy,  \,\,\,\,\,\,\,\,\,\, 
P_l(x) := \int_0^x p_l(y) dy \,\,\,\,\,\,\, \text{and }  \,\,\,\,\,\,\,
P(x) := \int_0^x p(y) dy.
\end{eqnarray*}
Observe that for all $x \in (0,\delta]$:
\begin{eqnarray*}
\probcnd{\text{The root $f$ of $T_{f,L}^{*}$ survives}}{\beta_{i+1}(f) < xn^{-1/2}} &=& \frac{P_{f,L}(x)}{x}, \\ 
\lim_{n \to \infty}
\probcnd{\text{The root $g$ of $T_l$ survives}}{\beta_{i+1}(g) < xn^{-1/2}} &=& \frac{P_{l}(x)}{x}, \\ 
\lim_{n \to \infty}
\probcnd{\text{The root $g$ of $T_\infty$ survives}}{\beta_{i+1}(g) < xn^{-1/2}} &=& \frac{P(x)}{x}.
\end{eqnarray*}

The next lemma, when combined with the discussion above and the definition of
$\A_{F,L}$, implies Lemma~\ref{lemma}.
\begin{lemma}
\label{lemma:final}
\mbox{}
\begin{enumerate}
\item[(i)]
$P(\delta)
 = \frac{\Phi((i+1)\delta) - \Phi(i\delta)}{\phi(i\delta)}$ and
$p(\delta) = \frac{\phi((i+1)\delta)}{\phi(i\delta)}$.
\item[(ii)] For all $x \in [0, \delta]$, $p_L(x) = p(x) (1 \pm o(\Gamma(i)\gamma(i)))$.
\item[(iii)] For all $x \in [0, \delta]$, $p_{f,L}(x) = p_L(x)(1 \pm 3\Gamma(i)\gamma(i))$. 
\end{enumerate}
\end{lemma}
The proof of Lemma~\ref{lemma:final} is given in the
next three subsections.

\subsection{Proof of Lemma~\ref{lemma:final}~(i)} \label{sec:branchingI}
%
Clearly $p(0) = 1$ and $P(0) = 0$. Hence, from the definition of survival and
the definition of $p(x)$ and $P(x)$, we get that for every $x \in [0,\delta]$,
at the limit as $n \to \infty$,
\begin{eqnarray}
\label{eq:g0}
p(x) &=& \bigg(1 - \frac{P(x)^2}{m^2}
         \bigg)^{m^2 \phi(i\delta)^2} \,
       \bigg(1 - \frac{P(x)}{m}\bigg)^{2m\Phi(i\delta)\phi(i\delta)}   \\
\nonumber
     &=& \exp\Big( - P(x)^2 \phi(i\delta)^2 -
                    2P(x) \Phi(i\delta)\phi(i\delta) 
             \Big).
\end{eqnarray}
By the fundamental theorem of calculus, $p(x)$ is the derivative of
$P(x)$.  Hence, we view~(\ref{eq:g0}) as the separable differential
equation that it is. This equation has the following as an implicit solution:
\begin{eqnarray*}
\int \exp\big(
   P^2 \phi(i\delta)^2 +
   2 P \phi(i\delta) \Phi(i\delta) 
         \big) dP = x.
\end{eqnarray*}
Solving the above integral, we get
\begin{eqnarray}
\label{eq:g3}
\frac{\sqrt{\pi}}{2} \, \text{erfi}\big(\Phi(i\delta) + \phi(i\delta) P\big)
           = x + C.
\end{eqnarray}
With the initial condition $P(0) = 0$, we get from~(\ref{eq:g3}) that
\begin{eqnarray*}
\label{eq:C}
\frac{\sqrt{\pi}}{2} \,  \text{erfi}(\Phi(i\delta)) = C.
\end{eqnarray*}
%
Let $z \ge 0$ satisfy
\begin{eqnarray*}
\exp( 
    - z^2 \phi(i\delta)^2 
    - 2z \phi(i\delta) \Phi(i\delta)
    ) = \frac{\phi((i+1)\delta)}{\phi(i\delta)}.
\end{eqnarray*}
A simple analysis shows that
\begin{eqnarray*}
z = \frac{\Phi((i+1)\delta) - \Phi(i\delta)}{\phi(i\delta)}.
\end{eqnarray*}
Taking $P=z$ and $C = \frac{\sqrt{\pi}}{2} \erfi(\Phi(i\delta))$, 
we solve~(\ref{eq:g3}) for $x$ to get
\begin{eqnarray*}
x = \frac{\sqrt{\pi}}{2} \, \text{erfi}\big(\Phi(i\delta) + \phi(i\delta)P\big) - C  
  = \frac{\sqrt{\pi}}{2} \, \big(\text{erfi}(\Phi((i+1)\delta)) - \text{erfi}(\Phi(i\delta))\big) = \delta,
\end{eqnarray*}
where the last equality is by the fact that $\frac{\sqrt{\pi}}{2} \erfi(\Phi(x)) = x$.  Hence, $P(\delta) =
\frac{\Phi((i+1)\delta) - \Phi(i\delta)}{\phi(i\delta)}$ and $p(\delta) =
\frac{\phi((i+1)\delta)}{\phi(i\delta)}$.
This completes the proof.  
\begin{remark}
\label{remark:r2}
As a side note, we observe that $0 \le P(\delta) \le
\delta$. Hence we get from the above conclusion and from Fact~\ref{fact:f1}
that $\frac{\Phi((i+1)\delta) - \Phi(i\delta)}{\delta} = P(\delta)\phi(i\delta)/\delta
\ge 0$ and that
$\frac{\Phi((i+1)\delta) - \Phi(i\delta)}{\delta} = P(\delta)\phi(i\delta)/\delta
\le 1$.
\end{remark}

\subsection{Proof of Lemma~\ref{lemma:final}~(ii)} \label{sec:branchingII}
%
Assume first that $L$ is odd.  Let $g_0$ be the root of $T_L$ and $T_\infty$.
Further assume $\beta_{i+1}(g_0) = xn^{-1/2}$ for some $x \in [0, \delta]$.
Clearly if $g_0$ survives in $T_L$ then $g_0$ survives in $T_\infty$.  Hence
$p_L(x) \le p(x)$.  Below we show that $p_L(x) \ge p(x) - n^{-36\epsilon}$. We
claim that this last inequality implies $p_L(x) = p(x) (1 - o(\Gamma(i)
\gamma(i)))$, which gives the lemma. Indeed, using the fact that $x \le \delta$
and since trivially $P(x) \le x$, it follows from~(\ref{eq:g0}), the
definition of $\delta$ and Fact~\ref{fact:f1} that $p(x) \sim 1$.  In addition, by Fact~\ref{fact:f1}
we have that $\Gamma(i)\gamma(i) = \Omega(n^{-35\epsilon})$.  Therefore we get,
as needed, 
\begin{eqnarray*}
p_L(x) \ge p(x) (1 - n^{-36\epsilon}/p(x)) = p(x) (1 - o(\Gamma(i)\gamma(i))).
\end{eqnarray*}

Say that a node $g$ at even distance from the root of $T_L$ is \emph{relevant},
if $g$ and its sibling (if exists) have a smaller birthtime than their
grandparent, and in addition, their grandparent is either relevant or the root.
Observe that if the root of $T_\infty$ survives then either the root of $T_L$
survives, or else, there is a relevant leaf in $T_L$.  It remains to show that
the expected number of relevant leaves in $T_L$ is at most $n^{-36\epsilon}$.

Say that a leaf $g_L$ in $T_L$ is a \emph{$c$-type} if the path leading from
the root to $g_L$ contains exactly $c$ nodes $G$ at odd distance from the root,
which are sets of size $1$. Consider a path $(g_0, G_1, g_1, \ldots, G_L, g_L)$
from the root to a leaf $g_L$, where $g_L$ is a $c$-type. Let $\G$ be the union
of $\{g_j: j \in [L]\}$ together with the set $\{g : \text{$g$ is a sibling of
some $g_j$, $j \in [L]$}\}$.  Since $g_L$ is a $c$-type, we have $|\G| = 2L -
c$. Now if $g_L$ is relevant, then for every node $g \in \G$,
$\event{\beta_{i+1}(g) < \beta_{i+1}(g_0) = xn^{-1/2}}$ holds. This event
occurs with probability $(x/m)^{2L-c}$.  Hence, the probability that $g_L$ is
relevant is at most
\begin{eqnarray*}
\Big(\frac{x}{m}\Big)^{2L-c} =
\Big(\frac{x}{m}\Big)^{c} \,
\Big(\frac{x^2}{m^2}\Big)^{L-c}.
\end{eqnarray*}
The number of $c$-type leaves in $T_L$ is at most
\begin{eqnarray*}
2^{L} \, (2m\Phi(i\delta)\phi(i\delta))^c \,
(2m^2\phi(i\delta)^2)^{L-c} \le
(4m\ln n)^c \, (4m^2)^{L-c},
\end{eqnarray*}
where the inequality is by Fact~\ref{fact:f1}.
Hence, the expected number of relevant $c$-type leaves in $T_L$ is at most
\begin{eqnarray*}
\Big(\frac{x}{m}\Big)^{c} \,
\Big(\frac{x^2}{m^2}\Big)^{L-c} \,
(4m\ln n)^c \, (4m^2)^{L-c} \le (4x \ln n)^{2L-c}.
\end{eqnarray*}
Now, $(4x \ln n)^{2L-c} \le \delta^{2L-c}(4\ln n)^{2L-c} \le \delta^{L-1} \sim
n^{-40\epsilon}$, where the inequalities are by $x \le \delta$, $c \le L$ and
$(4\ln n)^{2L} \le \delta^{-1}$.  To complete the proof, note that if a leaf is
a $c$-type, then we have at most $L + 1 = O(1)$ possible choices for $c$.
Therefore, with the union bound we conclude that the expected number of
relevant leaves in $T_L$ is at most $n^{-36\epsilon}$.

Next assume that $L$ is even, let $g_0$ be as above and assume
$\beta_{i+1}(g_0) = xn^{-1/2}$.  The proof for this case is similar to the
previous case and so we only outline it. It is easy to verify that if $g_0$
doesn't survive in $T_L$ then $g_0$ doesn't survive in $T_\infty$. Hence
$p_L(x) \ge p(x)$. Now, if $g_0$ doesn't survive in $T_\infty$ then either the
root of $T_L$ doesn't survive, or else, there is a relevant leaf in $T_L$.  One
can now show using the same argument as above that the expected number of
relevant leaves in $T_L$ is at most $n^{-36\epsilon}$.  This completes the
proof.

\subsection{Proof of Lemma~\ref{lemma:final}~(iii)} \label{sec:branchingIII}
%
The following implies Lemma~\ref{lemma:final}~(iii).
\begin{proposition}
\label{proposition:clm000}
Let $x \in [0,\delta]$, $0 \le l \le L$.
Let $g$ be a node at height $2l$ in $T_{f,L}^{*}$. Then
\begin{eqnarray*}
p_{g,l}(x) = p_l(x)(1 \pm 3\Gamma(i)\gamma(i)).
\end{eqnarray*}
\end{proposition}
\begin{proof}
The proof is by induction on $l$. The assertion holds for the base case since
by definition, $p_{g,0}(x) = p_0(x) = 1$ for all $x \in [0,\delta]$.  Let $1
\le l \le L$ and assume that the proposition holds for $l-1$.  Fix $x \in
[0,\delta]$ and let $g$ be a node at height $2l$ in $T_{f,L}^{*}$. 

For brevity, define $\eta := \Gamma(i)\gamma(i)$.  Further, let
\begin{eqnarray*}
Q^* :=
&& \Big(1 - \frac{P_{l-1}(x)(1-3\eta) }{m}\Big)^{2m\Phi(i\delta)\phi(i\delta)(1-1.02\Gamma(i))} \cdot \\
&& \Big(1 - \frac{P_{l-1}(x)^2(1-3\eta)^2}{m^2}\Big)^{m^2\phi(i\delta)^2(1-1.02\Gamma(i))}
\end{eqnarray*}
and
\begin{eqnarray*}
Q_* :=
&& \Big(1 - \frac{P_{l-1}(x)(1+3\eta)}{m}\Big)^{2m\Phi(i\delta)\phi(i\delta)(1+1.02\Gamma(i))} \cdot \\ 
&& \Big(1 - \frac{P_{l-1}(x)^2(1+3\eta)^2}{m^2}\Big)^{m^2\phi(i\delta)^2(1+1.02\Gamma(i))}.
\end{eqnarray*}
Let $g'$ be a grandchild of $g$. By the induction hypothesis and
by definition of $P_{g',l-1}(x)$ and $P_{l-1}(x)$, 
\begin{eqnarray*}
P_{g',l-1}(x) = P_{l-1}(x)(1 \pm 3\eta).
\end{eqnarray*}
Thus, it follows from the
definition of survival and by~(\ref{eq:child1})~and~(\ref{eq:child2}) that
\begin{eqnarray*}
Q_* \le p_{g,l}(x) \le Q^*.
\end{eqnarray*}
It remains to bound $Q^*$ and $Q_*$. In what follows we use the fact that
\begin{eqnarray}
\label{eq:exp}
\forall z > 1. \,\,\, \exp(-1/(z-1)) < 1-1/z < \exp(-1/z).
\end{eqnarray}
To bound $Q^*$, we have
\begin{eqnarray*}
 \Big(1 - \frac{P_{l-1}(x)(1-3\eta)}{m}\Big)^{2m\Phi(i\delta)\phi(i\delta)}  
&\le& 
 \Big(1 - \frac{P_{l-1}(x)}{m}\Big)^{2m\Phi(i\delta)\phi(i\delta)(1 - O(\eta))(1-1/m)} \\
&\le&
 \Big(1 - \frac{P_{l-1}(x)}{m}\Big)^{2m\Phi(i\delta)\phi(i\delta)(1 - O(\eta))} \\
&\le&
 \Big(1 - \frac{P_{l-1}(x)}{m}\Big)^{2m\Phi(i\delta)\phi(i\delta)} 
 \Big(1 - \frac{\delta}{m}\Big)^{-O(m\Phi(i\delta)\phi(i\delta)\eta)} \\
&\le&
  \exp\big(-2P_{l-1}(x)\Phi(i\delta)\phi(i\delta)\big)
  \cdot (1 + o(\eta)),
\end{eqnarray*}
where the first inequality follows from~(\ref{eq:exp}); the second inequality
follows from the fact that $1/m = o(\eta)$, which in turn follows from the
definition of $m$ and from Fact~\ref{fact:f1}; the third inequality follows
since $P_{l-1}(x) \le x \le \delta$; and the last inequality follows
from~(\ref{eq:exp}) and Fact~\ref{fact:f1}. 
For similar reasons we also have that
\begin{eqnarray*}
 \Big(1 - \frac{P_{l-1}(x)^2(1-3\eta)^2}{m^2}\Big)^{m^2\phi(i\delta)^2}  
&\le& 
 \Big(1 - \frac{P_{l-1}(x)^2}{m^2}\Big)^{m^2\phi(i\delta)^2(1 - O(\eta))(1-1/m)} \\
&\le&
 \Big(1 - \frac{P_{l-1}(x)^2}{m^2}\Big)^{m^2\phi(i\delta)^2(1 - O(\eta))} \\
&\le&
 \Big(1 - \frac{P_{l-1}(x)^2}{m^2}\Big)^{m^2\phi(i\delta)^2} 
 \Big(1 - \frac{\delta^2}{m^2}\Big)^{-O(m^2\phi(i\delta)^2\eta)} \\
&\le& 
 \exp\big(-P_{l-1}(x)^2\phi(i\delta)^2\big)
 \cdot (1 + o(\eta)).
\end{eqnarray*}
In addition, since $P_{l-1}(x)(1-3\eta) \le x \le \delta$, and by definition
of $\gamma(i)$, we have
\begin{eqnarray*}
\Big(1 - \frac{P_{l-1}(x)(1-3\eta)}{m}\Big)^{-2m\Phi(i\delta)\phi(i\delta) \cdot 1.02\Gamma(i)} 
&\le& \Big(1 - \frac{\delta}{m}\Big)^{-2m\Phi(i\delta)\phi(i\delta) \cdot 1.02 \Gamma(i)} 
\le 1  + 2.05\eta,  \\
\Big(1 - \frac{P_{l-1}(x)^2(1-3\eta)^2}{m^2}\Big)^{-m^2\phi(i\delta)^2 \cdot 1.02\Gamma(i)} 
&\le& \Big(1 - \frac{\delta^2}{m^2}\Big)^{-m^2\phi(i\delta)^2 \cdot 1.02 \Gamma(i)} 
\le 1  + 1.03\eta.
\end{eqnarray*}
Then by the fact that
\begin{eqnarray*}
p_l(x) =
\exp\big( -P_{l-1}(x)^2\phi(i\delta)^2 -2P_{l-1}(x)^2\Phi(i\delta)\phi(i\delta) \big), 
\end{eqnarray*}
we can conclude that
\begin{eqnarray*}
Q^* \le p_l(x)  (1 + 3\eta).
\end{eqnarray*}
The argument for the lower bound on $Q_*$ is similar.
\end{proof}

\subsection{When $2m\Phi(i\delta)\phi(i\delta)$ isn't an integer}
\label{sec:last}
%
We have defined the tree $T_\infty$ so that for every node $g$ at even distance
from the root, the number of children of $g$ that are sets of size $1$ is
exactly $\ceil{2m\Phi(i\delta)\phi(i\delta)}$.  We further made the simplifying
assumption that $2m\Phi(i\delta)\phi(i\delta)$ is an integer. Reviewing our
proof above, we needed this simplifying assumption in order to get a relatively
simple solution to the differential equation in Section~\ref{sec:branchingI}.
Here we briefly explain how one can modify the proof above so as to handle the
case where $2m\Phi(i\delta)\phi(i\delta)$ is not an integer.  

The first step would be to take a random subtree of $T_{f,L}^*$.  Let $\zeta
\in [0.1, 0.9]$ be such that $\zeta \cdot 2m\Phi(i\delta)\phi(i\delta)$ is an
integer. Keep every subtree of $T_{f,L}^{*}$ that is rooted at a set of
size~$1$ with probability~$\zeta$. This gives us a random subtree of
$T_{f,L}^*$.  From now on we only care about this random subtree and so for
brevity, we refer to this subtree by $T_{f,L}^*$. Using the fact that $\E^*$
holds, one can show the following. With probability $1 - n^{-\omega(1)}$, for
every non-leaf node $g$ at even distance from the root of $T_{f,L}^{*}$, the
number of children of $g$ that are sets of size~$1$ is $\zeta \cdot 2m
\Phi(i\delta) \phi(i\delta) (1 \pm 1.02\Gamma(i))$ and the number of children
of $g$ that are sets of size~$2$ is as given in~(\ref{eq:child2}).
Given that, we change the definition of $T_\infty$ accordingly by asserting
that for every node $g$ at even distance from the root of $T_\infty$, the
number of children of $g$ that are sets of size $1$ is $\zeta \cdot
2m\Phi(i\delta)\phi(i\delta)$.

Having redefined the above trees, the second step is to redefine the
distribution of the birthtimes of the edges in $T_{f,L}^*$ and $T_\infty$.  The
birthtime of an edge that appears in a set of size~$1$ in $T_{f,L}^{*}$ or in
$T_\infty$ is redefined so that it is distributed uniformly at random in $[0,
\zeta \cdot mn^{-1/2}]$, whereas the birthtime of an edge that appears in a set
of size~$2$ in $T_{f,L}^{*}$ or in $T_\infty$ remains uniformly distributed at
random in $[0, mn^{-1/2}]$ as before. 

The rest of the proof is straightforward. In particular, the statement of
Lemma~\ref{lemma:final} is not changed. The only necessary other modifications
are the obvious ones that follow from the above changes in the definition of
$T_{f,L}^*$ and $T_\infty$ and the definition of the birthtimes of the edges
that appear in those trees.

\section*{Acknowledgment}
The author would like to kindly thank Prof. Joel Spencer for suggesting the
problem of estimating the number of small subgraphs 
in $\TF(n,p)$.



\appendix

\section{Proof of Fact~\ref{fact:f1}} \label{app:1}
%
Recall that $\frac{\sqrt{\pi}}{2} \erfi(\Phi(x)) = x$, where $\erfi(x)$ is the
imaginary error function, given by, for example, $\erfi(x) =
\frac{2}{\sqrt{\pi}} \sum_{j = 0}^{\infty} \frac{x^{2j+1}}{j!(2j+1)}$.  We have
that $\erfi(x) \to \exp(x^2)/(\sqrt{\pi}x)$ as $x \to \infty$. Hence, it
follows that as $x \to \infty$, $\Phi(x) \to \sqrt{\ln x}$ and $\phi(x) \to
(2x\sqrt{\ln x})^{-1}$.

\begin{enumerate}
\item[(i)]
We first upper bound $\phi(i\delta)$ and $\Phi(i\delta)$.  We have that
$\erfi(x) \ge 0$ if and only if $x \ge 0$.  By the fact that
$\frac{\sqrt{\pi}}{2} \erfi(\Phi(x)) = x$ we have $\erfi(\Phi(i\delta)) = 2 i
\delta / \sqrt{\pi} \ge 0$.  Hence $\Phi(i\delta) \ge 0$.  Therefore
$\phi(i\delta) = \exp(-\Phi(i\delta)^2) \le 1$. Next, note that $\erfi(x)$ is
monotonically increasing with $x$. We also have by $\frac{\sqrt{\pi}}{2}
\erfi(\Phi(x)) = x$ that $\erfi(\Phi(i\delta))$ is monotonically increasing
with $i$. Hence $\Phi(i\delta)$ is monotonically increasing with $i$ and so
$\Phi(i\delta) \le \Phi(I\delta)$. The upper bound on $\Phi(i\delta)$ now
follows since $I\delta \sim n^{\epsilon}$ and so $\Phi(I\delta) \sim
\sqrt{\ln n^\epsilon}$.

Next, we lower bound $\phi(i\delta)$ and $\Phi(i\delta)$ (for $i \ge 1$).
Since $\Phi(i\delta)$ is monotonically increasing with $i$, we have that
$\phi(i\delta)$ is monotonically decreasing with $i$.  Therefore, it remains to
show that $\phi(I\delta) = \Omega(\delta^{1.5})$ and $\Phi(\delta) =
\Omega(\delta)$.  The fact that $\phi(I\delta) = \Omega(\delta^{1.5})$ follows
since $\phi(I\delta) \to 1/(2I\delta\sqrt{\ln I\delta})$.  The fact that
$\Phi(\delta) = \Omega(\delta)$ follows directly from the fact that
$\frac{\sqrt{\pi}}{2} \erfi(\Phi(x)) = x$ and the definition of $\erfi(x)$.  

\item[(ii)] By~(i) we have $\delta \Phi(i\delta)\phi(i\delta) \le \delta \ln n
= o(1)$ and $\delta^2 \phi(i\delta)^2 \le \delta^2 = o(1)$. Hence $\gamma(i) =
o(1)$.  It also follows directly from the definition of $\gamma(i)$ and from
the previous item that $\gamma(i) = \Omega(\delta^5)$.

We now bound $\Gamma(i)$. Since $\Gamma(i)$ is monotonically non-decreasing and
$\Gamma(0) = n^{-30\epsilon}$, it is enough to show that $\Gamma(I) \le
n^{-10\epsilon}$.  We do that by first showing that
$\Gamma(\delta^{-1}\floor{\ln\ln n}) \le n^{-30\epsilon + o(1)}$. For brevity,
we shall assume below that $\floor{\ln\ln n} = \ln \ln n$.

For every $0 \le i \le \delta^{-1}\ln\ln n$, $\Phi(i\delta) \le \ln\ln n$
(crudely) and $\phi(i\delta) \le 1$.  Therefore, we have that for every $0 \le
i \le \delta^{-1}\ln\ln n$, 
\begin{eqnarray*}
\delta \Phi(i\delta) \phi(i\delta) &\le& \delta \ln\ln n, \text{ and} \\
\delta^2\phi(i\delta)^2 &\le& \delta \ln\ln n.
\end{eqnarray*}
Hence, for $0 \le i \le \delta^{-1}\ln\ln n$, $\gamma(i) \le \delta \ln\ln n$ and so
\begin{eqnarray*}
\Gamma(\delta^{-1} \ln\ln n) \le n^{-30\epsilon}  (1 + 10\delta \ln\ln n)^{\delta^{-1}\ln\ln n} 
  = n^{-30\epsilon + o(1)}.
\end{eqnarray*}
Now, note that for every $\delta^{-1} \ln\ln n \le i \le I$,
\begin{eqnarray*}
\delta \Phi(i\delta) \phi(i\delta) &\le& 0.6/i, \text{ and} \\
\delta^2 \phi(i\delta)^2 &\le& 0.6/i,
\end{eqnarray*}
and this follows from the fact that for $\delta^{-1} \ln\ln n \le i \le I$,
$\Phi(i\delta) \phi(i\delta) \sim 1/(2 i\delta)$ and $\phi(i\delta) \le 1/(2
i\delta)$.  Hence, for $\delta^{-1} \ln\ln n \le i \le I$, $\gamma(i) \le
0.6/i$ and so we conclude that
\begin{eqnarray*}
\Gamma(I) \le n^{-30\epsilon + o(1)} \prod_{1 \le i \le I} (1 + 6/i)  
  \le n^{-30\epsilon + o(1)} \cdot \exp(7\ln I)  \le n^{-10\epsilon}.
\end{eqnarray*}
\end{enumerate}


\begin{bibdiv}
\begin{biblist}

\bib{AlonSpencer}{book}{
   author={Alon, N.},
   author={Spencer, J.},
   title={The probabilistic method},
   edition={2nd ed.},
   publisher={Wiley, New York},
   year={2000},
}


\bib{Bohman}{article}{
      author={Bohman, T.},
       title={The triangle-free process},
        date={2009},
     journal={Advances in Mathematics},
        note={To appear},
}
\bib{BKeevash}{article}{
      author={Bohman, T.},
      author={Keevash, P.},
       title={The early evolution of the $H$-free process},
        note={To appear},
}


\bib{Boll00}{article}{
      author={Bollob{\'a}s, B{\'e}la},
      author={Riordan, Oliver},
       title={Constrained graph processes},
        date={2000},
     journal={Electr. J. Comb.},
      volume={7},
}




\bib{ErdosSW95}{article}{
      author={Erd\H{o}s, Paul},
      author={Suen, Stephen},
      author={Winkler, Peter},
       title={On the size of a random maximal graph},
        date={1995},
     journal={Random Struct. Algorithms},
      volume={6},
      number={2/3},
       pages={309\ndash 318},
}
\bib{Hoeffding}{article}{
      author={Hoeffding, W.},
       title={Probability inequalities for sums of bounded random variables},
        date={1963},
     journal={J. Amer. Statist. Assoc.},
      volume={58},
       pages={13--30},
}
\bib{uppertail}{article}{
      author={Janson, Svante},
      author={Rucinski, Andrzej},
       title={The infamous upper tail},
        date={2002},
     journal={Random Struct. Algorithms},
      volume={20},
      number={3},
       pages={317--342},
}
\bib{Kim}{article}{
      author={Kim, Jeong Han},
       title={The Ramsey number $R(3,t)$ has order of magnitude $t^2/\log t$},
        date={1995},
     journal={Random Structures and Algorithms},
      volume={7},
       pages={173--207},
}
\bib{McD}{article}{
      author={McDiarmid, C.},
       title={On the method of bounded differences},
        date={1989},
     journal={Surveys in Combinatorics (Proceedings, Norwich 1989)},
      series={London Math. Soc. Lecture Note Set. 141},
   publisher={Cambridge Univ. Press, Cambridge},
       pages={148--188},
}
\bib{OsthusT01}{article}{
      author={Osthus, Deryk},
      author={Taraz, Anusch},
       title={Random maximal h-free graphs},
        date={2001},
     journal={Random Struct. Algorithms},
      volume={18},
      number={1},
       pages={61\ndash 82},
}


\bib{RucinskiW92}{article}{
      author={Rucinski, Andrzej},
      author={Wormald, Nicholas~C.},
       title={Random graph processes with degree restrictions},
        date={1992},
     journal={Combinatorics, Probability {\&} Computing},
      volume={1},
       pages={169\ndash 180},
}


\bib{Spencer0a}{article}{
      author={Spencer, Joel~H.},
       title={Maximal triangle-free graphs and {R}amsey $r(3,t)$},
        date={1995},
        note={Unpublished manuscript},
}

\bib{SpencerPrivate}{misc}{
      author={Spencer, Joel~H.},
       title={Private communication},
        date={July, 2008},
}

\bib{Wol}{article}{
      author={Wolfovitz, Guy},
       title={Lower bounds for the size of random maximal $H$-free graphs},
        date={2009},
     journal={Electr. J. Comb.},
      volume={16},
}

\end{biblist}
\end{bibdiv}
\end{document}